\documentclass[11pt, a4paper]{amsart}

\usepackage{amscd,verbatim}
\usepackage{amssymb}

\usepackage{bm}

\usepackage{amssymb, amsmath}
\usepackage[all]{xy}
\usepackage[inline]{enumitem}
\usepackage{hyphenat}
\usepackage[colorlinks,linkcolor=blue,citecolor=blue,urlcolor=red]{hyperref}
\usepackage[dvipsnames]{xcolor}
\usepackage{mathtools}
\usepackage{tikz-cd}
\usetikzlibrary{cd,decorations.pathreplacing}
\usepackage[labelformat=empty]{caption}
\usepackage{xfrac}
 \usepackage[ left=3cm, right=3cm, top = 2.8cm, bottom = 2.8cm ]{geometry}
\usepackage{stmaryrd}
\usepackage{bbm}
\usepackage{enumitem}
\makeatletter
\def\@tocline#1#2#3#4#5#6#7{\relax
  \ifnum #1>\c@tocdepth % then omit
  \else
    \par \addpenalty\@secpenalty\addvspace{#2}%
    \begingroup \hyphenpenalty\@M
    \@ifempty{#4}{%
      \@tempdima\csname r@tocindent\number#1\endcsname\relax
    }{%
      \@tempdima#4\relax
    }%
    \parindent\z@ \leftskip#3\relax \advance\leftskip\@tempdima\relax
    \rightskip\@pnumwidth plus4em \parfillskip-\@pnumwidth
    #5\leavevmode\hskip-\@tempdima
      \ifcase #1
      \or\or \hskip 2em \or \hskip 2em \else \hskip 3em \fi%
      #6\nobreak\relax
    \dotfill\hbox to\@pnumwidth{\@tocpagenum{#7}}\par
    \nobreak
    \endgroup
  \fi}
\makeatother

\newcommand{\A}{\mathbf{A}}

\newcommand{\C}{\mathbf{C}}

\newcommand{\G}{\mathbf{G}}

\newcommand{\N}{\mathbb{N}}
\renewcommand{\P}{\mathbf{P}}
\newcommand{\Q}{\mathbb{Q}}
\newcommand{\R}{\mathbb{R}}
\newcommand{\Z}{\mathbb{Z}}
\newcommand{\F}{\mathbb{F}}

\newcommand{\sD}{\mathcal{D}}

\newcommand{\sM}{{\mathscr M}}
\newcommand{\sO}{\mathcal{O}}
\newcommand{\sP}{\mathcal{P}}

\newcommand{\bZ}{\mathbb{Z}}
\newcommand{\fm}{\mathfrak{m}}

\newcommand{\fp}{\mathfrak{p}}

\newcommand{\fX}{\mathfrak{X}}
\newcommand{\fY}{\mathfrak{Y}}

\newcommand{\Xb}{{\overline{X}}}

\newcommand{\Cone}{\operatorname{Cone}}

\newcommand{\Ker}{\operatorname{Ker}}

\newcommand{\Spec}{\operatorname{Spec}}
\newcommand{\Spa}{\operatorname{Spa}}
\newcommand{\Spf}{\operatorname{Spf}}
\newcommand{\Proj}{\operatorname{Proj}}

\newcommand{\Sch}{\operatorname{\mathbf{Sch}}}
\newcommand{\Shv}{\operatorname{\mathbf{Shv}}}
\newcommand{\Ab}{\operatorname{\mathbf{Ab}}}

\newcommand{\pro}[1]{\text{\rm pro}_{#1}\text{\rm--}}

\newcommand{\red}{{\operatorname{red}}}

\newcommand{\Zar}{{\operatorname{Zar}}}
\newcommand{\Nis}{{\operatorname{Nis}}}
\newcommand{\et}{{\operatorname{\acute{e}t}}}

\newcommand{\inj}{\hookrightarrow}

\newcommand{\id}{{\operatorname{Id}}}

\newcommand{\ch}{{\operatorname{ch}}}
\newcommand{\Sym}{{\operatorname{Sym}}}

\newcommand{\pr}{{\operatorname{pr}}}
\newcommand{\Frac}{{\operatorname{Frac}}}

\renewcommand{\lim}{\operatornamewithlimits{\varprojlim}}
\newcommand{\colim}{\operatornamewithlimits{\varinjlim}}

\newcommand{\ol}{\overline}

\renewcommand{\phi}{\varphi}
\renewcommand{\epsilon}{\varepsilon}

\newcommand{\CI}{\operatorname{\mathbf{CI}}}

\newcommand{\Bl}{{\mathbf{Bl}}}

\newcommand{\M}{\mathbf{M}}

\def\rmapo#1{\overset{#1}{\longrightarrow}}

\def\bZ{\mathbb{Z}}

\def\tV{\widetilde{V}}
\def\tW{\widetilde{W}}

\newcounter{spec}
{\end{list}}%

\newtheorem{lemma}{Lemma}[section]
\newtheorem{thm}[lemma]{Theorem}

\newtheorem{theorem}{Theorem}
\newtheorem{cor-intro}{Corollary}
\newtheorem{lem-intro}{Lemma}
\newtheorem{prop}[lemma]{Proposition}

\newtheorem{corollary}[lemma]{Corollary}

\newtheorem{claim}[lemma]{Claim}
\theoremstyle{definition}
\newtheorem{defn}[lemma]{Definition}

\newtheorem{definition}[lemma]{Definition}
\newtheorem{para}[lemma]{}
\theoremstyle{remark}

\newtheorem{rmk}[lemma]{Remark}

\newtheorem{ex}[lemma]{Example}

\numberwithin{equation}{section}

\numberwithin{equation}{lemma}

\colorlet{LightRubineRed}{RubineRed!70!}

\usepackage[color = blue!20, bordercolor = black, textsize = tiny]{todonotes}

\usepackage{relsize} 
\usepackage[bbgreekl]{mathbbol} 
\usepackage{amsfonts} 
\DeclareSymbolFontAlphabet{\mathbb}{AMSb} %to ensure that the meaning of \mathbb does not change
\DeclareSymbolFontAlphabet{\mathbbl}{bbold} 

\DeclareSymbolFontAlphabet{\mathbbl}{bbold}

\usepackage{enumitem}
\def\Ab{\mathbf{Ab}}
\def\lSm{\mathbf{lSm}}
\def\SmlSm{\mathbf{SmlSm}}
\def\Sm{\mathbf{Sm}}
\def\sM{\mathcal{M}}
\def\RSC{\mathbf{RSC}}
\def\RSCet{\RSC_{\et}}
\def\septet{\mathrm{septet}}
\def\smur{\mathrm{smur}}
\def\LtY{\Lambda_{\tY}}
\makeatletter
\newcommand{\mylabel}[2]{#2\def\@currentlabel{#2}\label{#1}}
\makeatother

 \def\k{\kappa}
\def\tX{\tilde{X}}
\def\tx{\tilde{x}}
\def\ty{\tilde{y}}
\def\tU{\tilde{U}}
\def\tV{\tilde{V}}
\def\tW{\tilde{W}}
\def\tT{\tilde{T}}
\def\tA{\tilde{A}}
\def\tB{\tilde{B}}
\def\tC{\tilde{C}}
\def\tD{\tilde{D}}
\def\ta{\tilde{a}}
\def\tb{\tilde{b}}
\def\tc{\tilde{c}}
\def\tf{\tilde{f}}
\def\tphi{\tilde{\phi}}
\def\tg{\tilde{g}}
\def\th{\tilde{h}}
\def\tp{\tilde{p}}
\def\tpr{\tilde{pr}}
\def\tq{\tilde{q}}

\def\XSt{(X/S)_{t}}
\def\Xdt{X_{dt}}
\def\Ftsh{F_t^\sharp}
\def\Xkdt{(X/k)_{dt}}
\def\tY{\tilde{Y}}
\def\Xb{\ol{X}}
\def\tu{\tilde{u}}
\def\tV{\tilde{V}}
\def\cU{\mathcal{U}}
\def\cW{\mathcal{W}}
\def\cV{\mathcal{V}}
\def\cT{\mathcal{T}}
\def\cY{\mathcal{Y}}
\def\cQ{\mathcal{Q}}
\def\Spa{\mathrm{Spa}}
\def\W{\mathbb{W}}
\def\tP{\tilde{P}}
\def\Za{Z_\alpha}
\def\kv{\kappa(v)}
\def\catprojlim#1{\underset{#1}{``\varprojlim"}}
\def\catinjlim#1{\underset{#1}{``\varinjlim"}}
\def\projlim#1{\underset{#1}{\varprojlim}}
\def\injlim#1{\underset{#1}{\varinjlim}}

\def\qaq{\;\text{ and }\;}
\def\rmapou#1#2{\underset{#2}{\overset{#1}{\longrightarrow}}}
\def\rmapo#1{\overset{#1}{\longrightarrow}}
\def\lmapo#1{\overset{#1}{\longleftarrow}}

\def\rig{\mathrm{rig}}
\def\Fhrig{\hat{F}^{\rig}}
\def\fXrig{\mathfrak{X}^{\rig}}
\def\OK{\cO_K}
\def\ShvCoh{\Shv((X,\tX)_t)_{\sO^{int}-coh}}
\def\qfor{\text{ for }}
\def\SchSS{\Sch_{(S,\tS)}}
\def\SchSSt{\Sch_{(S,\tS),t}}
 \def\k{\kappa}
 \def\Sp{\mathrm{Sp}}
\def\supn#1{|{#1}|_{\sup}}
\def\htX{\widehat{\tX}}
\def\htY{\widehat{\tY}}
\def\bB{\mathbb{B}}
\def\tor{\mathrm{tor}}

\def\Xb{\overline{X}}

\def\FKC#1{\textbf{#1}}

\newcounter{elno}   

\begin{document}

\def\THH{\operatorname{THH}}
\def\TC{\operatorname{TC}}
\def\TCmin{\operatorname{TC}^-}
\def\TP{\operatorname{TP}}
\def\HH{\operatorname{HH}}
\def\HC{\operatorname{HC}}
\def\HCmin{\operatorname{HC}^-}
\def\HP{\operatorname{HP}}

\def\Fil{\operatorname{Fil}}
\def\Gr{\operatorname{Gr}}
\def\gr{\operatorname{gr}}

\def\QSyn{\operatorname{QSyn}}
\def\QRSPerfd{\operatorname{QRSPerfd}}
\def\lQSyn{\operatorname{lQSyn}}
\def\lQRSPerfd{\operatorname{lQRSPerfd}}

\def\syn{\mathrm{syn}}
\def\Fsyn{\mathrm{Fsyn}}
\def\Fet{\mathrm{F\acute{e}t}}

\def\LogRec{\operatorname{\mathbf{LogRec}}}
\def\Ch{\operatorname{\mathrm{Ch}}}

\def\ltr{\mathrm{ltr}}

\def\kX{\mathfrak{X}}
\def\kY{\mathfrak{Y}}

\def\otCIsp{\otimes_{\CI}^{sp}}
\def\otCINissp{\otimes_{\CI}^{\Nis,sp}}

\def\tL{\tilde{L}}
\def\tX{\tilde{X}}
\def\tY{\tilde{Y}}
\def\tT{\tilde{T}}
\def\tF{\widetilde{F}}
\def\tG{\widetilde{G}}
\def\tR{\widetilde{R}}
\def\tS{\widetilde{S}}

\def\Sh{\operatorname{\mathbf{Shv}}}
\def\PSh{\operatorname{\mathbf{PSh}}}
\def\Shltr{\operatorname{\mathbf{Shv}_{dNis}^{ltr}}}
\def\Shlog{\operatorname{\mathbf{Shv}_{dNis}^{log}}}
\def\Shvlog{\operatorname{\mathbf{Shv}^{log}}}
\def\Sm{\operatorname{\mathrm{Sm}}}
\def\SmlSm{\operatorname{\mathrm{SmlSm}}}
\def\lSm{\operatorname{\mathrm{lSm}}}
\def\FlSm{\operatorname{\mathrm{FlSm}}}
\def\FlQSm{\operatorname{\mathrm{FlQSm}}}
\def\FlQSyn{\operatorname{\mathrm{FlQSyn}}}
\def\lCor{\operatorname{\mathrm{lCor}}}
\def\SmlCor{\operatorname{\mathrm{SmlCor}}}
\def\PShltr{\operatorname{\mathbf{PSh}^{ltr}}}
\def\PShlog{\operatorname{\mathbf{PSh}^{log}}}
\def\logCI{\mathbf{logCI}} 

\def\Mod{\operatorname{Mod}}

\def\Log{\operatorname{\mathcal{L}\textit{og}}}
\def\Rsc{\operatorname{\mathcal{R}\textit{sc}}}
\def\Pro{\mathrm{Pro}\textrm{-}}
\def\pro{\mathrm{pro}\textrm{-}}
\def\dg{\mathrm{dg}}
\def\plim{\mathrm{``lim"}}
\def\ker{\mathrm{ker}}
\def\coker{\mathrm{coker}}
\def\PrL{\mathcal{P}\mathrm{r^L}}
\def\PrLo{\mathcal{P}\mathrm{r^{L,\otimes}}}
\def\Spt{\mathcal{S}\mathrm{pt}}
\def\PSpt{\mathrm{Pre}\mathcal{S}\mathrm{pt}}
\def\Fun{\mathrm{Fun}}
\def\Sym{\mathrm{Sym}}
\def\CAlg{\mathrm{CAlg}}
\def\Poly{\mathrm{Poly}}
\def\Cat{\mathrm{Cat}}

\def\Alb{\operatorname{Alb}}
\def\bAlb{\mathbf{Alb}}
\def\Gal{\operatorname{Gal}}

\def\hofib{\mathrm{hofib}}
\def\fib{\mathrm{fib}}
\def\triv{\mathrm{triv}}
\def\ABl{\mathcal{A}\textit{Bl}}
\def\divsm#1{{#1_\mathrm{div}^{\mathrm{Sm}}}}

\def\cA{\mathcal{A}}
\def\cB{\mathcal{B}}
\def\cC{\mathcal{C}}
\def\cD{\mathcal{D}}
\def\cE{\mathcal{E}}
\def\cF{\mathcal{F}}
\def\cG{\mathcal{G}}
\def\cH{\mathcal{H}}
\def\cI{\mathcal{I}}
\def\cJ{\mathcal{J}}
\def\cK{\mathcal{K}}
\def\cL{\mathcal{L}}
\def\cM{\mathcal{M}}
\def\cN{\mathcal{N}}
\def\cO{\mathcal{O}}
\def\cP{\mathcal{P}}
\def\cQ{\mathcal{Q}}
\def\cR{\mathcal{R}}
\def\cS{\mathcal{S}}
\def\cT{\mathcal{T}}
\def\cU{\mathcal{U}}
\def\cV{\mathcal{V}}
\def\cW{\mathcal{W}}
\def\cX{\mathcal{X}}
\def\cY{\mathcal{Y}}
\def\cZ{\mathcal{Z}}

\def\tcA{\widetilde{\mathcal{A}}}
\def\tcB{\widetilde{\mathcal{B}}}
\def\tcC{\widetilde{\mathcal{C}}}
\def\tcD{\widetilde{\mathcal{D}}}
\def\tcE{\widetilde{\mathcal{E}}}
\def\tcF{\widetilde{\mathcal{F}}}

\def\one{\mathbbm{1}}

\def\XP{X \backslash \sP}
\def\M0a{{}^t\cM_0^a}
\newcommand{\Ind}{{\operatorname{Ind}}}

\def\Xkbar{\overline{X}_{\overline{k}}}
\def\dx{{\rm d}x}

\newcommand{\dNis}{\operatorname{dNis}}
\newcommand\loget{{\operatorname{l\acute{e}t}}}
\newcommand\ket{{\operatorname{k\acute{e}t}}}
\newcommand\vet{{\operatorname{v\acute{e}t}}}
\newcommand\ABNis{{\operatorname{AB-Nis}}}
\newcommand\sNis{{\operatorname{sNis}}}
\newcommand\sZar{{\operatorname{sZar}}}
\newcommand\set{{\operatorname{s\acute{e}t}}}
\newcommand\cofib{\mathrm{Cofib}}

\newcommand{\Gmlog}{\G_m^{\log}}
\newcommand{\Gmlogred}{\overline{\G_m^{\log}}}

\newcommand{\varcolim}{\mathop{\mathrm{colim}}}
\newcommand{\varlim}{\mathop{\mathrm{lim}}}
\newcommand{\tensor}{\otimes}

\newcommand{\eq}[2]{\begin{equation}\label{#1}#2 \end{equation}}
\newcommand{\eqalign}[2]{\begin{equation}\label{#1}\begin{aligned}#2 \end{aligned}\end{equation}}

\def\varplim#1{\text{``}\varlim_{#1}\text{''}}
\def\det{\mathrm{d\acute{e}t}}

\renewcommand{\int}{\mathrm{int}}

\def\tor{\mathrm{tor}}
\def\Op{\cO^+}
\def\Omp#1{\Omega^{#1,+}}
\def\rig{\mathrm{rig}}
\def\an{\mathrm{an}}
\def\tXrig{\tX^{\rig}}
\def\Omprig#1{\Omega^{#1,\rig}}
\def\Gar{\G_a(r)}
\def\Gas{\G_a(s)}
\def\Gasr{\G_a(s,r)}
\def\Gardv{\G_a(r)^{dv}}
\def\Garrig{\G_a(r)^\rig}
\def\Gasrrig{\G_a(sr)^\rig} 
\def\Gaz{\G_a(0)}
\def\Gazdv{\G_a(0)^{dv}}
\def\hK{\hat{K}}
\def\hR{\hat{R}}
\def\hX{\hat{X}}
\def\htA{\hat{\tA}}
\def\hA{\hat{A}}
\def\htX{\hat{\tX}}
\def\htY{\hat{\tY}}
\def\htZ{\hat{\tZ}}
\def\hX{\hat{X}}
\def\hGar{\widehat{\Gar}}
\def\Sp{{\mathrm Sp}}
\newcommand{\prolim}[1]{\underset{#1}{\operatorname{``} {\rm lim} \operatorname{''}  \hspace{.2ex}  }}  
\def\tAT{\tA[T]}

\def\tATz{\tA[T]^z_\pi}
\def\ATz{A[T]^z_\pi}
\def\RTz{R[T]^z_\pi}
\def\KTz{K[T]^z_\pi}
\def\tAz{\tA^z_\pi}
\def\AtXz{(\A_{\tX}^1)^z_\pi}
\def\PtXz{(\P_{\tX}^1)^z_\pi}
\def\AXz{(\A_{X}^1)^z_\pi}
\def\tXz{\tX^z_\pi}
\def\Xz{X^z_\pi}
\def\Az{A^z_\pi}

\def\tATh{\tA[T]^h_\pi}
\def\ATh{A[T]^h_\pi}
\def\tAh{\tA^h_\pi}
\def\AtXh{(\A_{\tX}^1)^h_\pi}
\def\PtXh{(\P_{\tX}^1)^h_\pi}
\def\AXh{(\A_{X}^1)^h_\pi}
\def\tXh{\tX^h_\pi}
\def\Xh{X^h_\pi}
\def\Ah{A^h_\pi}

\def\tZ{\tilde{Z}}
\def\rig{\mathrm{rig}}
\def\orig{\mathrm{o,rig}}
\def\htAT{\htA\langle T\rangle}
\def\hAT{\hA\langle T\rangle}

\def\th{\tilde{h}}
\def\tAp{\tA_{\fp}}
\def\modpi{\;\mathrm{mod}^\times \:\pi}
\def\tth{\tilde{\theta}}
\def\tAm{\tA_{\fm}}
\def\kp{\kappa(\fp)}
\def\tg{\tilde{g}}
\def\tUb{\overline{U}}
\def\bu{\mathrm{Bl}}
\def\PShoc{\PSh^{\mathrm{oc}}}
\def\Shvoc{\Shv^{\mathrm{oc}}}
\def\Gaz{\G_a(0)}
\def\zpzcs{(\Z/p\Z)_0}
\def\Frob{\mathrm{Frob}}
\def\jZ{j_{\cZ}}
  \def\rhob{\overline{\rho}}
  \def\kfp{\kappa(\fp)}
  
 \def\kp{\kappa(\fp)}
  
\def\Or{\cO(r)}
\def\Ors{\cO(s,r)}
\def\Orss#1{\cO_{#1}(s,r)}
\def\ATc{A\langle T\rangle}

\def\Oc{\sO^{\circ}}
\def\Sp{\mathrm{Sp}}
 \def\sO{\cO}
 \def\bB{{\mathbb B}}
\def\Fa{F_a}

\def\Yrig{Y^{\rig}}
\def\Xrig{X^{\rig}}
\def\Zrig{Z^{\rig}}

\def\zar{\mathrm{zar}}
\def\rig{\mathrm{rig}}
\def\OtX{\cO_{\tX}}
\def\OtY{\cO_{\tY}}
\def\OtU{\cO_{\tU}}
\def\tor{\mathrm{tor}}
\def\Fzar#1{F_{{#1}_{\zar}}}
\def\Fet#1{F_{{#1}_{\et}}}
\def\StX{\Sigma_{\tX}}
\def\St(#1){\Sigma_{#1}}
\def\OStX{\cO_{\StX}}
\def\Rpi{R[1/\pi]}
\def\QCoh{\mathbf{QCoh}}
\def\Coh{\mathbf{Coh}}
\def\ShvQCoh#1{\Shv^{\mathrm{qcoh}}(#1_t)}
\def\ShvCoh#1{\Shv^{\mathrm{coh}}(#1_t)}
\def\ShvaCoh#1{\Shv^{\mathrm{qc,acoh}}(#1_t)}
\def\Shvfp#1{\Shv^{\mathrm{fp}}(#1_t)}
\def\fY{\mathfrak{Y}}
\def\SX{\Sigma_{\fX}}
\def\StX{\Sigma_{\tX}}
\def\StY{\Sigma_{\tY}}
\def\StXp{\Sigma'_{\tX}}
\def\LtX{\Lambda_{\tX}}
\def\LtY{\Lambda_{\tY}}
\def\Fhrig{\widehat{F}^{\rig}}
\def\Valsept#1{\mathrm{Val}^{\mathrm{sept}}_{#1}}

 \def\ZRX{\langle\fX\rangle}
 \def\RZ{\mathrm{RZ}}
\def\OcRZ#1{{\sO^{\circ}_{\RZ(#1)}}}
\def\OcRZX{{\sO^{\circ}_{\RZ(\fX)}}}
 \def\sp{\mathrm{sp}}
\def\fXrig{\fX^{\rig}}
\def\fYrig{\fY^{\rig}}
\def\Bl{\mathrm{Bl}}
\def\fU{\mathfrak{U}}
\def\OK{\cO_K}
\def\sOt{\sO^t}

\def\Vals{\mathrm{Val}^{\mathrm{s}}}

 \def\OKt#1#2{\Omega^{#1,t}_{#2/K}}
 \def\OKtt#1{\Omega^{#1,t}_{-/K}}
 \def\OFt#1#2{\Omega^{#1,t}_{#2/F}}
 \def\OFtt#1{\Omega^{#1,t}_{-/F}}
 
 \def\OKtrigg#1{\Omega^{#1,o}_{/K}}
 \def\OFtrigg#1{\Omega^{#1,o}_{/F}}
 \def\OKtrig#1#2{\Omega^{#1.o}_{#2/K}}
 \def\OFtrig#1#2{\Omega^{#1,o}_{#2/F}}
 \def\qfor{\text{ for }\;}
 \def\qwith{\;\text{ with }\;}
 \def\vZar{\mathrm{vZar}}
  \def\fib{\mathrm{fib}}
  
\def\SmSS{\Sm_{(S,\tS)}}
\def\SmSSt{\Sm_{(S,\tS),t}}
\def\SchSS{\Sch_{(S,\tS)}}
\def\SchSSt{\Sch_{(S,\tS),t}}

\def\supn#1{|{#1}|_{\sup}}

\def\Ab{\mathbf{Ab}}
\def\lSm{\mathbf{lSm}}
\def\SmlSm{\mathbf{SmlSm}}
\def\Sm{\mathbf{Sm}}
\def\sM{\mathcal{M}}
\def\RSC{\mathbf{RSC}}
\def\RSCet{\RSC_{\et}}
\def\septet{\mathrm{septet}}
\def\smur{\mathrm{smur}}
\def\LtY{\Lambda_{\tY}}
\makeatletter
%\newcommand{\mylabel}[2]{#2\def\@currentlabel{#2}\label{#1}}
%\makeatother

 \def\k{\kappa}
\def\tX{\tilde{X}}
\def\tx{\tilde{x}}
\def\ty{\tilde{y}}
\def\tU{\tilde{U}}
\def\tV{\tilde{V}}
\def\tW{\tilde{W}}
\def\tT{\tilde{T}}
\def\tA{\tilde{A}}
\def\tB{\tilde{B}}
\def\tC{\tilde{C}}
\def\tD{\tilde{D}}
\def\ta{\tilde{a}}
\def\tb{\tilde{b}}
\def\tc{\tilde{c}}
\def\tf{\tilde{f}}
\def\tphi{\tilde{\phi}}
\def\tg{\tilde{g}}
\def\th{\tilde{h}}
\def\tp{\tilde{p}}
\def\tpr{\tilde{pr}}
\def\tq{\tilde{q}}

\def\XSt{(X/S)_{t}}
\def\Xdt{X_{dt}}
\def\Ftsh{F_t^\sharp}
\def\Xkdt{(X/k)_{dt}}
\def\tY{\tilde{Y}}
\def\Xb{\ol{X}}
\def\tu{\tilde{u}}
\def\tV{\tilde{V}}
\def\cU{\mathcal{U}}
\def\cW{\mathcal{W}}
\def\cV{\mathcal{V}}
\def\cT{\mathcal{T}}
\def\cY{\mathcal{Y}}
\def\cQ{\mathcal{Q}}
\def\Spa{\mathrm{Spa}}
\def\W{\mathbb{W}}
\def\tP{\tilde{P}}
\def\Za{Z_\alpha}
\def\kv{\kappa(v)}
\def\catprojlim#1{\underset{#1}{``\varprojlim"}}
\def\catinjlim#1{\underset{#1}{``\varinjlim"}}
\def\projlim#1{\underset{#1}{\varprojlim}}
\def\injlim#1{\underset{#1}{\varinjlim}}
\def\qaq{\;\text{ and }\;}
\def\rmapo#1{\overset{#1}{\longrightarrow}}
\def\lmapo#1{\overset{#1}{\longleftarrow}}

\author{Alberto Merici}
\thanks{Funded by the Deutsche Forschungsgemeinschaft (DFG, German Research 
Foundation)
TRR 326 \textit{Geometry and Arithmetic of Uniformized Structures}, 
project number 444845124./}
\address{Institut f\"ur Mathematik, Universit\"at Heidelberg. INF 205, 69120 Heidelberg, Germany}
\email[A. Merici]{merici@mathi.uni-heidelberg.de}

\author{Kay R\"ulling}
\address{School of Mathematics and Natural Sciences, University of Wuppertal, Germany}
\email[K. R\"ulling]{ruelling@uni-wuppertal.de}

\author{Shuji Saito}
\address{Graduate School of Mathematical Sciences, the University of Tokyo, 3-8-1 Komaba Meguro-ku
Tokyo 153-8914, Japan}
\email[S. Saito]{sshuji.goo@gmail.com}

\title[Birational lattices in the cohomology of the structure sheaf]{Birational and 
\texorpdfstring{${\A^1}$}{A1}-invariant lattices in the cohomology of the structure sheaf  over non-archimedean fields}

\begin{abstract} 
We show that the cohomology of the structure sheaf of smooth and proper schemes over a complete non-archimedean field $K$ of characteristic zero, can be refined to an $\A^1$-invariant cohomology theory of smooth (not necessarily proper) schemes over $K$ with values in $\cO_K$-lattices, and the same holds for $K$ of positive characteristic in dimensions at most $3$. 
As one application, we obtain that the automorphism group of the function field of a proper smooth variety $X$ of dimension at most 3 
over a field of positive characteristic acts quasi-unipotently on the cohomology of the structure sheaf of $X$. 
The construction of the lattices relies on a variant of the tame cohomology of Hübner--Schmidt with coefficients in a twisted version 
of the  tame structure sheaf and uses results from rigid analytic geometry on the cohomology of 
twisted integral rigid structure sheaves due to Bartenwerfer and van der Put.
\end{abstract}
\maketitle
\tableofcontents

\section*{Introduction}

It is a classical question whether a given cohomology theory  $X\mapsto H(X)$ defined on some category of schemes 
with values in finite dimensional vector spaces over a 
complete non-archimedean field $K$  admits a functorial lattice over the ring of integers $\sO_K$.
A positive answer to this question directly implies that for any endomorphism of relevant schemes $f:X\to X$, 
the characteristic polynomial of the induced $K$-linear endomorphism  $f^*$ acting on $H(X)$ has coefficients in $\sO_K$.

For example, as we will see below how integral structures on the cohomology of the structure sheaf of smooth proper $K$-schemes
can be defined using a variant of the tame site of Hübner--Schmidt. It was pointed out to us by Bhatt that an integral structure can as well be defined analytically by first using GAGA to go to rigid analytic varieties and then applying a classical result of Bartenwerfer on the cohomology of the rigid integral structure sheaf on smooth affinoids. 

In case $K$ has characteristic zero, it is well-known that by Hironaka's work the cohomology of the structure sheaf of {smooth proper $K$-schemes} extends
to a functor on the category $\Sm_K$ of all smooth separated $K$-schemes by the formula
\[X\mapsto R\Gamma(\ol{X}, \cO),\]
where $\ol{X}$ is some fixed smooth compactification of $X$ over $K$. In case the characteristic of $K$ is positive,  
the same is true if one restricts to smooth separated $K$-schemes of dimension at most $3$, where one uses 
\cite{CossartPiltant-ResArith3folds} to get the existence of a smooth compactification and 
\cite{CR15} to get the independence of such a compactification.  Note that this gives an $\A^1$-invariant cohomology theory.
The main purpose of this paper is to show that the above cohomology theory on $\Sm_K$ admits an algebraically defined 
$\A^1$-invariant integral structure which satisfies tame descent.

More precisely, we prove the following.

\begin{theorem}\label{thm;main}
Let $K$ {be a complete non-archimedean field} with $R=\OK$ the ring of integers. 
Let $\Mod^f_R$ be  the full subcategory of $R$-modules that are of finite type up to bounded torsion (see Definition \ref{def;ModR}) and $\sD^{\geq0}(R)^f$ be the full subcategory of the derived category of bounded-below complexes  
of $R$-modules whose cohomology are in $\Mod_R^f$. 
Assume $K$ is either a discrete valuation field or algebraically closed.
 \begin{enumerate}[label=(\arabic*) ]
\item\label{thm;main2} 
If $\ch(K)=0$\footnote{But the characteristic of the residue field of $R$ may be positive.}, there is a functor 
\[ \cF: \Sm_K \to \sD^{\geq0}(R)^f,\quad X\mapsto \cF(X)\]
enjoying the following properties for $X\in \Sm_K$: 

 \begin{enumerate}[label=(\alph*)]
 \item\label{thm;main2a}
 There is a natural equivalence
 \[\cF(X)\otimes_R K\simeq R\Gamma(\ol{X},\cO)\]
 for every smooth compactification $\Xb$ of $X$ over $K$.
 For a map $\Xb'\to \Xb$ of such compactifcations, the above equivalence is   compatible in an obvious sense. 
   \item\label{thm;main2b}($\A^1$-invariance) 
  $\cF(X)\simeq \cF(X\times_K \A^1_K)$.
  \item\label{thm;main2c}(Birational invariance)
$\cF(X) \simeq \cF(Y)$ for any map $Y\to X$ in $\Sm_K$ which is an isomorphism over a dense open subset of $X$.   
  \item\label{thm;main2d}(Tame descent)
  For   $Y\to X$ an $R$-tame covering in the sense of \cite{HS2020}, we have an equivalence
 \[ \cF(X) \simeq \varprojlim
\Big(\cF(Y) \begin{smallmatrix}  \longrightarrow\\ \longleftarrow\\  \longrightarrow \end{smallmatrix}\cF(Y\times_X Y)
\begin{smallmatrix}  \longrightarrow\\\longleftarrow\\  \longrightarrow\\ \longleftarrow \\ \longrightarrow
  \end{smallmatrix}\cF(Y\times_X Y\times_X Y) \cdots\Big)\]
\end{enumerate}
\item If $\ch(K)>0$, there exists 
\[ (\Sm_K^{\leq 3})^{op}\to \cD^{\ge 0}(R)^f, \quad X\mapsto \cF(X),\]
enjoying the same properties as above except that we assume $\dim(X)\leq 2$ in \ref{thm;main2b}, where $\Sm_K^{\leq 3}$ is the full subcategory of $\Sm_K$ consisting of those $X$ with $\dim(X)\leq 3$. 
\end{enumerate}
 \end{theorem} 
 
 Recall (see \cite{HS2020}) that a morphism of schemes $f:Y\to X$ is an $R$-tame covering if it is an \'etale covering and for any $x\in X$ and a valuation $v$ on $\kappa(x)$ trivial over $R$, there exists $y\in Y$ lying over $x$ and a valuation $w$ on $\kappa(y)$ extending $v$ such that $\cO_w/\cO_v$ is tame, i.e. $[\Frac(\cO_w^{sh}):\Frac(\cO_v^{sh})]$ is prime to the exponential characteristic of the residue field of $\cO_v$, where $(-)^{sh}$ denotes the strict henselization.
 \medbreak

 We remark that the only reason why we need to work with smooth schemes in Theorem \ref{thm;main} is to use Bartenwerfer's result on cohomology of smooth affinoids over $K$ (see \cite[Folgerung~3]{B2} and also the last paragraph of the introduction). 
If the latter result can be generalized to affinoids with certain prescribed singularities  (e.g. rational, log canonical, Du Bois), then we can relax the smoothness assumption to such singularities in the above theorem, see Theorem \ref{thm;main-sing} for a precise statement.

\medbreak

As an application  of Theorem \ref{thm;main}, we obtain the following.

\begin{theorem}\label{thm;application}
Let $k$ be an arbitrary field and let $\Lambda \subseteq k$ be the integral closure of $\bZ$.
Let $X$ be a smooth and proper $k$-scheme. Assume either $\ch(k)=0$ or $\dim(X)\leq 3$. For a $k$-rational map $\phi: X\dasharrow X$ and $i\in \N$, the characteristic polynomials ${\rm det}\big(T-\phi^*|H^i(X,\cO_X)\big)$ and ${\rm det}\big(T-\phi^*|H^i(X,\omega_X)\big)$ lie in $\Lambda[T]$.
\end{theorem}
\def\tGam{\tilde{\Gamma}}

\medbreak

The action of the $k$-rational map $\phi$ on $H^i(X,\cO_X)$ is defined as follows: Take a dense open $U\subset X$ and $g: U \to X$ representing $\phi$. Let $\Gamma_\phi\subset X\times_k X$ be the closure of the graph $\Gamma_g\subset U\times_k X$ and choose $\tGam\to \Gamma_\phi$, a resolution of singularities, which exists thanks to \cite{Hironaka} and \cite{CossartPiltant-ResArith3folds} by the assumption $\dim(X)\leq 3$.
Let $\pr_i:\tGam\to X$ be the two projections. Then, $\phi^*$ is defined to be the composite
\[ H^i(X,\cO_X) \rmapo{\pr_2^*} H^i(\tGam,\cO_{\tGam}) \rmapou{(\pr_1^*)^{-1}}{\simeq}  H^i(X,\cO_X) ,\]
where $\pr_1^*$ is an isomorphism, see e.g. \cite[Theorem 1.1]{CR15}, noting that $\pr_1$ is an isomorphism over $U$. 
Theorem \ref{thm;application} is reduced to the case $k$ is finitely generated over its prime field.
Noting that $\Lambda$ is the intersection of the completions of the valuation rings of all discrete valuations on $k$, it is then reduced to showing 
${\rm det}\big(T-\phi^*|H^i(X,\cO_X)\big)\in \OK[T]$ if $K$ is a complete discrete valuation field. 
Then it follows by Theorem \ref{thm;main}\ref{thm;main2c} and Serre duality.

\medbreak

We remark that in case $\ch(k)=0$, 
Theorem \ref{thm;application} can also be deduced from  Hodge theory: By the Lefschetz principle, we may assume $k=\C$.
Then, a rational map $\phi$ induces a map on the cohomology $H^i(X(\C)^{an},\C)$ via a correspondence action, which preserves the lattice $H^i(X(\C)^{an},\Z)$. By Hodge theory we have that ${\rm det}\big(T-\phi^*|H^i(X,\cO_X)\big)$ is a factor of 
${\rm det}\big(T-\phi^*|H^i(X(\C)^{an},\C)\big)$, 
 which lies in $\Z[T]$.
 \medbreak

In positive characteristic, we get the following corollary.
 
\begin{cor-intro}\label{cor:application}
Let $k$ be a field of characteristic $p>0$ and let $X$ be a smooth {and proper $k$-scheme of dimension}  $\le 3$.
Then  ${\rm Aut}_k(k(X))$ the group of birational $k$-automorphisms of $X$ acts quasi-unipotently on 
$H^i(X, \sO_X)$, for all $i$.    
\end{cor-intro}
Here we note that in general an automorphism of the function field  $\varphi\in {\rm Aut}_k(k(X))$ defines a morphism  
\[\varphi^*:=\Gamma^*_{\varphi}: H^i(X, \sO_X)\to H^i(X, \sO_X),\] 
by considering $\Gamma_{\varphi}$ the closure in $X\times_k X$
of the graph of the map $\Spec k(X)\to \Spec k(X)$ induced by $\varphi$. This is a well-defined action by \cite[Theorem 3.1.8 and Proposition 3.2.2]{CR11}, in case $k$ is perfect, 
and by \cite[Theorem 4.17 and Theorem 5.1]{Amazeen} in general.\footnote{Here \cite[Theorem 4.17]{Amazeen} ensures that we have a well-defined morphism $\Gamma_{\varphi}^*$ and \cite[Proposition 5.1]{Amazeen} ensures that $\Gamma_{\varphi}^*\circ \Gamma_{\psi}^*=\Gamma_{\varphi\circ\psi}^*$.}
To prove the corollary, we note that $X$ and any automorphism of its function field are defined over some finitely generated field
and by base change we may henceforth assume that $k$ is finitely generated over $\F_p$. Thus the integral closure of $\F_p$ in 
$k(X)$ is a finite field and Theorem \ref{thm;application} implies that all eigenvalues of $\varphi^*$ acting on $H^i(X, \sO_X)$ 
lie in $\F_q^\times$, for some $q=p^r$, and hence are $(q-1)$st roots of unity.

\medbreak

 \subsection*{Construction of $\cF$.}
Now we explain how to construct the functor $\cF$ from Theorem \ref{thm;main} using a variant of \emph{tame} cohomology introduced by  H\"ubner and Schmidt \cite{HS2020}.
We remark that it is possible to construct such $\cF$ satisfying \ref{thm;main2a}, \ref{thm;main2b} and \ref{thm;main2c} without use of tame cohomology, only using classical rigid-analytic theorems of Bartenwerfer and van der Put, while \ref{thm;main2d} will follow from a comparison theorem of  the cohomology of certain sheaves on the tame site and cohomology of rigid spaces which is viewed as an integral  GAGA theorem, as explained below.

 Let $K\supset R$ be as in Theorem \ref{thm;main} and put $\tS=\Spec(R)$ and $S=\Spec(K)$. 
To construct such $\cF$ as in Theorem \ref{thm;main}, we consider the category $\SchSS$ whose objects are pairs $(U,\tU)$ of open immersions $U\hookrightarrow \tU$ over $\tS$ such that $U\to \tS$ factors through $S\hookrightarrow \tS$.
We then equip $\SchSS$ with the \emph{tame} topology introduced by  H\"ubner and Schmidt \cite{HS2020} modified in \cite{MRS-1}, which has the covering families recalled in \ref{def;tametopology-intro}\ref{def;tametopology-intro4} below.  
The corresponding site is denoted by $\SchSSt$. 
For $(X,\tX)\in \SchSS$, let $(X,\tX)_t$ be the site whose underlying category is the category of objects $(U,\tU)$ over $(X,\tX)$ with $U\to X$ \'etale, endowed with the induced topology.    

For $X$ a separated $S$-scheme of finite type, choose a Nagata compactification $X\hookrightarrow \tX$ of $X\to \tS$ and define 
for a sheaf of abelian groups $F$ on $\SchSSt$
 \[ R\Gamma_t(X/R,F) =R\Gamma((X,\tX)_t,F_{|(X,\tX)_t}).\]
Then, it turns out that $R\Gamma_t(X/R,F)$ does not depend on the choice of $\tX$ and extends to a functor 
 \[ R\Gamma_t(-/R,-): (\Sch_K)^{op}\times \Sh^{\rm ab}(\SchSSt)  \to \cD(\Z).\]
In Theorem \ref{thm;main} we choose  $\cF$ as 
\eq{eq;cFr}{
 \Sm_K\ni X\mapsto \cF(X) := R\Gamma_t(X/R,\G_a(0)),}
with $\G_a(0)$  given by  
\eq{eq;Ga0intr}{ \G_a(0)(U,\tU)= \Big\{a\in \cO(U)|\; a\in  \sqrt{\pi \sO(\tU^{\rm  int})}\Big\},}
where $\pi$ is a fixed pseudo uniformizer of $R$ and $\tU^{\int}$ is the integral closure of $\tU$ in $U$.
We remark that there are other possible choices for $\cF$ which essentially change the resulting lattice by multiplication with 
some power of $\pi$. The main features of $\G_a(0)$, which are non-trivial, 
are that it defines a sheaf in the tame topology 
and  when restricted to  $\tX_{\et}$ is a quasi-coherent sheaf and is {\em almost coherent} 
in the sense of Zavyalov when further restricted to $\tX_{\Zar}$ and $K$ is algebraically closed. 
Moreover, if $\fXrig$ denotes the rigid space over $K$ associated to the formal completion $\fX$ of $\tX$ 
along the special fiber, then the rigidification $\G_a(0)_{\fXrig}$ is the sheaf $\cO(1)$  on $\fXrig$
considered by Bartenwerfer and van der Put. This is the sheaf given on an affinoid subdomain 
$U=\Sp(B)\subset \fXrig$ by 
\[\cO(1)(U)=\{f\in B\mid |f|_{\sup}<1\},\] 
where $|-|_{\sup}$ is the sup norm on $B$, i.e.,
 $\supn f = \sup_{x\in \Sp(B)} |f(x)|$. See Lemma \ref{lem:Gar3} and Corollary \ref{cor:Gar}. We also remark 
that for our purposes, it does not suffice to consider the tame structure sheaf $\cO^t$ 
given by $\cO^t(U,\tU)= \sO(\tU^{\rm  int})$ instead as its rigidification 
is the  rigid integral structure sheaf $\cO^o$ and it is not known whether the cohomological vanishing results which we need also hold for $\cO^o$. This is related to the fact that
the quotient $\sO/\sO^o$ of the rigid structure sheaf by the integral rigid structure sheaf 
{does not seem to be overconvergent} in the sense of \cite{dJvdP}, whereas $\sO/\sO(1)$ is.
See below for more details.

 \medbreak

 The main result of this article is that $\cF$ satisfies the properties of Theorem \ref{thm;main}. 
 The properties \ref{thm;main2}\ref{thm;main2a} and \ref{thm;main2d} follow directly from the 
 definition.
A key ingredient in the proof of \ref{thm;main2}\ref{thm;main2b} and \ref{thm;main2c} and the fact that {$\cF$ takes values in $\cD^{\geq 0}(R)^f$} is the following comparison theorem. 

\begin{theorem}[Theorem \ref{thm;comparison-final}]\label{thm;comparison-thm1}
Let $R\subset K$ be as above. Let $X$ be a normal separated $K$-scheme of finite type and 
let $X\hookrightarrow \tX$ be a compactification over $R$. 
Let $\fXrig$ be the rigid space over $K$ associated to the formal completion $\fX$ of $\tX$ along the special fiber. 

If $X$ is {\em proper} over $K$ , then there exists a canonical equivalence
\[  R\Gamma_t(X/R,\G_a(0))\simeq R\Gamma(\fXrig,\cO(1)).\]
 
The same holds if $X$ is smooth (not necessarily proper) over $K$ and $\ch(K)=0$ or $\dim X\le 3$.
\end{theorem} 

In particular, if $X$ is smooth and if $\ch(K)=0$ or $\dim(X)\leq 3$, we have an equivalence
\[  \cF(X)\simeq \cF(U),\]
for any dense open subscheme $U\subset X$, proving Theorem \ref{thm;main}\ref{thm;main2}\ref{thm;main2c} (see Remark \ref{rmk;thm;comparison-final}).
This  may be viewed as a birational invariance of $\A^1$-invariant sheaves, in the sense of \cite{KS}.

\medbreak

In the case $K$ is a complete discrete valuation field and $X$ is proper over $K$,
Theorem \ref{thm;comparison-thm1} is a special case of the following {\em integral} variant of GAGA (cf. \cite[\FKC{II}.9.4.2]{FK}). For $(X,\tX)\in \SchSS$ with $X=\tX\otimes_{\OK} K$, we introduce a subcategory
$\ShvCoh{(X,\tX)}$ of $\Shv((X,\tX)_t)$ consisting of $\cO^t$-modules on $(X,\tX)_t$, which are coherent as $\sO_{\tU}$-modules 
when restricted to any $(U, \tU)\in (X,\tX)_t$ and construct a functor
\[ \ShvCoh{(X,\tX)} \to \Shv (\fXrig), \quad F \to F_{\fXrig},\]
such that $ \G_a(0)_{\fXrig}=\cO(1)$ (see \S\ref{para:tame-qcoh} and Corollary \ref{cor:Gar}  for details).

\begin{theorem}[Theorem \ref{thm;comparison}]\label{thm;comparison-intro}
If $\tX$ is proper over $\OK$, there is an equivalence 
\[ R\Gamma((X,\tX)_t,F) \simeq R\Gamma(\fXrig,F_{\fXrig}),  \quad  
\text{for } F\in \ShvCoh{(X,\tX)}.\]
\end{theorem}

Using Zavyalov's theory of quasi-coherent almost coherent $\sO$-modules \cite{Zavyalov},
we get a similar result as above also in the case where $K$ is algebraically closed, see 
Theorem \ref{thm:almost-comparison}.
\medbreak

The end of the proof of \ref{thm;main2}\ref{thm;main2b} and \ref{thm;main2c} is now completely reduced to computations in classical rigid geometry, namely:
\begin{lem-intro}[Lemma \ref{lem1;CohRigid}, Lemma \ref{lem:hdi-general}]\label{lem1;CohRigid-intro}
Let $f^{\rig}:\P^{n,\rig}\times \fXrig\to \fXrig$ be the rigidification of the projection map $\P^n\times \tX\to \tX$.
Then, for any $s\in \R_{>0}$, 
\[Rf^{\rig}_*(\cO_{\P^{n,\rig}\times \fXrig}(s))\simeq \cO_{\fXrig}(s),\]
where the sheaf $\cO_{\cY}(s)$ on a rigid space $\cY$ over $K$ is given by
$\cO_{\cY}(s)(U)=\{f\in B\mid |f|_{\sup}<s\}$ on an affinoid subdomain 
$U=\Sp(B)\subset \cY$.
\end{lem-intro}

Using the base change theorem in rigid geometry (see \cite[Theorem 2.7.4]{dJvdP}), the lemma is deduced from the 
following vanishing due to Bartenwerfer \cite[Theorem]{B2} and van der Put \cite[Theorem 3.15]{vdP}
\[H^i(D,\sO(s))=0,\quad \text{for all $s>0$ and integers $i>0$},\]
where  $D\subset \bB^d_K=\Sp(K\langle z_1,\dots,z_d\rangle)$ is a generalized polydisc.
Again, using the base change theorem in rigid geometry the above is reduced to the case $d=1$, which is proved by using the Mittag-Leffler decomposition of analytic functions 
on affinoid subdomains of the unit disc.
We also note that the base change theorem in rigid geometry 
is only useful when applied 
to cohomology of overconvergent sheaves, 
which is another reason why we must work with $\G_a(0)$ from 
\eqref{eq;Ga0intr} instead of $\cO^t$.
In our understanding, it is this analytic computation which makes it possible to get around resolution of singularities over $R$.

\medbreak

Finally, the fact that the cohomology of $\cF(X)$ lies in $\Mod_R^f$ {for $X\in \Sm_K$} is proven in Theorem \ref{thm:fg-uncond}.
The proof uses  finiteness results of the cohomology of a proper $R$-scheme with 
(almost) coherent coefficients, and the following fact from rigid geometry:
for any smooth affinoid space $V$ over a non-archemidean field $K$ and  for any $i>0$, there exists $c\in K$ with $|c|<1$ such that 
$c\cdot H^i(V,\sO(s)) =0$ for all $s\in \R_{>0}$, which is proven by Bartenwerfer in  
\cite[Folgerung~3]{B2} (see also \cite[Theorem 17]{KST-BQ}).
The proof of this result is analytic. It uses the open mapping theorem for surjective linear maps of Banach spaces.

\medbreak

\subsection*{Future work}

In case $K$ is a cdvf and $\ch(K)=0$, we hope the construction of $\cF$ can be generalized 
to an integral structure on de Rham (and crystalline) cohomology. To this end one has to 
construct suitable sheaves  $\Omega^q(r)$ on the tame site $(X,\tX)_t$ 
which are coherent when restricted  to $\tX_{\Zar}$ and for which  a variant of Lemma 
\ref{lem1;CohRigid-intro} and of the vanishing \cite[Folgerung 3]{B2} referred to above 
can be proven. One candidate is $\Omega^{q,t}\otimes_{\sO^t}\G_a(0)$, but for the moment we do not know if this satisfies any of these properties.
\subsection*{Acknowledgments}
We thank Keiji Oguiso for his suggestion to consider rational maps in Theorem \ref{thm;application}. We thank Bhargav Bhatt for his comments on earlier versions, which in particular inspired the content of \S \ref{ss;sing}. We thank Andreas Bode, Tobias Schmidt, and Vasudevan Srinivas for useful discussions, and Daichi Takeuchi for providing a Hodge theoretic proof of Theorem \ref{thm;application} in case $\ch(k)=0$.

\section{Recollection}

In this section, we collect some results and constructions of \cite{MRS-1}.

 \def\k{\kappa}
\begin{para}\label{def;tametopology-intro}
Recall the definition of the tame site from \cite[Definition 2.4]{MRS-1}, which is an algebraic version of H\"ubner's adic-tame site \cite{Huebner2021}.
	\begin{enumerate}[label=(\arabic*)]
		\item\label{def;tametopology-intro1}
		Let $X\hookrightarrow \tX$ be a quasi-compact open immersion of qcqs schemes.
		Let $(X,\tX)_\tau$ be the category of pairs $(U,\tU)$ consisting of quasi-compact open immersions $U\to \tU$ of $\tX$-schemes such that $\tU\to \tX$ is integral over finite type (ift) (see \cite[\S 1]{MRS-1}) 
        and $U\to \tX$ factors through an \'etale morphism $U\to X$.
		Morphisms $(V,\tV)\to (U,\tU)$ are pairs of morphisms $f: V\to U$ in $X_{\et}$ and $\tf:\tV\to \tU$ a morphism of $\tX$-schemes satisfying the obvious compatibility.
		  \item\label{def;tametopology-intro2}
		For $ (U,\tU)\in (X,\tX)_\tau$, let $\Spa(U,\tU)$ be
        the set of triples $(x,v,\epsilon)$ such that $x\in U$, $v$ is a valuation on $\k(x)$ and $\epsilon\colon \Spec(\cO_v)\to \tU$ is a map compatible with $\Spec(\k(x))\to X$.
		\item\label{def;tametopology-intro3}
		A morphism $(f,\tf):(V,\tV)\to (U,\tU)$ in $(X,\tX)_\tau$ is \emph{tame over $(x,v,\epsilon_v)\in \Spa(U,\tU)$} if 
		there is $(y,w,\epsilon_w)\in \Spa(V,\tV)$ such that $f(y)=x$, $w_{|k(x)}=v$, and $w/v$ is tamely ramified and the following diagram commutes:\[
		\begin{tikzcd}
			\Spec(\cO_w)\ar[r,"\epsilon_w"]\ar[d,"(f)_{|y,w\geq 0}"]&\tV\ar[d,"\tf"]\\
			\Spec(\cO_v)\ar[r,"\epsilon_v"]&\tU.
		\end{tikzcd}
		\]
		If $(f,\tf)$ is a \emph{quasi-modification}, i.e.,  $f$ is an isomorphism and $\tf$ is universally closed, 
        then $(f,\tf)$ is tame at every $(x,v,\epsilon_v)\in \Spa(U,\tU)$, by the valuative criterion.    
		\item\label{def;tametopology-intro4}
		The \emph{tame} topology on $(X,\tX)_\tau$
		is generated by families $\{(f_i,\tf_i): (V_i,\tV_i) \to (U,\tU)\}_i$ of maps such that for every $(x,v,\epsilon_v)\in \Spa(U,\tU)$,
		there is $i\in I$ such that $(V_i,\tV_i) \to (U,\tU)$ is {tame} over $(x,v,\epsilon_v)$.
		The corresponding site is denoted by $(X,\tX)_t$. For  $F\in \Sh((X,\tX)_t)$ a sheaf of sets on 
        $(X,\tX)_t$, every quasi-modification $(X,\tY)\to (X,\tX)$ induces a bijection $F(X,\tX)\to F(X,\tY)$, 
        see \cite[Lemma 2.7]{MRS-1}.
        \item\label{def;tametopology-intro5} For every $(U,\tY)\in (X,\tX)_t$, we let $F_{(U,\tY)_{\et}}$ 
        (resp. $F_{(U,\tY)_{\Zar}}$) denote the sheaf on $\tY_{\et}$ (resp. $\tY_{\Zar}$) given by 
        \[
        \tV/\tY\mapsto F(U\times_{\tY}\tV,\tV).\]        
		\end{enumerate}
	\end{para}
	
\begin{thm}[{\cite[Theorem 7.2]{MRS-1}}]\label{thm:tame-vs-etale-intro}
	Let $X\inj \tX$ be a quasi-compact open immersion of qcqs schemes. 
	Let $F$ be a sheaf of abelian groups on $(X,\tX)_t$ such that the following condition is satisfied:
	\begin{enumerate}[label=(p)]
		\item\label{thm:tame-vs-etale1}
		for every $(U,\tU)\in (X,\tX)_t$ and $x\in \tU$, $F(\Spec(\sO_{\tU,x})\times_{\tU} U, \Spec(\sO_{\tU,x}))$ is a $\Z_{(p_x)}$-module, where $p_x$ is the exponential characteristic of $\kappa(x)$.
	\end{enumerate}
	Then, we have  canonical isomorphisms 
	\[
	H^i((U,\tU)_t, F)\cong \colim_{(U,\tY) \in \LtX(U)} H^i(\tY_{\et}, F_{(U,\tY)_{\et}})
    \cong \colim_{(U,\tX') \in \StX(U)} H^i(\tX'_{\et}, F_{(U,\tX')_{\et}}), \quad i\ge 0.
	\]
	Here the first colimit is indexed by the (cofiltered) category $\LtX(U)$ of all quasi-modifications $\cV\to \cU$ with $\cV$ integral, and the second colimit is indexed by the (cofiltered) category $\StX(U)$ of admissible blow-ups $(X,\tY)\to (X,\tX)$, in a coherent ideal $\cA\subset \cO_{\tX}$ such that the support of $\cO_{\tX}/\cA$ is contained in $\tX\setminus X$.
\end{thm}

\begin{para}\label{para:beta}
In \cite[\S 8.1]{MRS-1}, we give a recipe to construct tame sheaves. For $x\in X$ and $k(x)\subseteq L$ a separable extension, let $\cO_{X,L}^h$ be the henselization of $X$ at $L$, i.e. the limit of the \'etale neighborhoods $\Spec(L)\to U\to X$. For $F$ an \'etale sheaf on $X$, assume that we are given a family of subgroups\[
\beta:=\{F_w\subseteq F(\cO_{X,L}^h)\}_{(L,w)}
\] 
indexed over all finite tame extensions $(L,w)$ of points $(k(x),v,\epsilon)\in \Spa(X,\tX)$, such that if $(L',w')/(L,w)$ is a tame extension, $F_{w}\subseteq F(\cO_{X,L}^h)$ is mapped to $F_{w'}\subseteq F(\cO_{X,L'}^h)$. 

For $(U,\tU)\in (X,\tX)_t$ consider
\[F_{\beta}(U,\tU):=\left\{a\in F(U)\,\middle\vert\,
\begin{minipage}[c]{7cm} for all $(x,v,\epsilon)\in \Spa(U,\tU)$ 
	there exists  a finite tame extension $(L,w)/ (k(x), v)$, such that  $a_L\in F_w$\end{minipage}\right\},\]
where $a_L$ denotes the pullback of $a\in F(U)$ along  $\Spec \sO^h_{X,L}\to U$. 
\end{para}

\begin{prop} (see \cite[Proposition 8.2]{MRS-1})
	The assignment $(U,\tU)\mapsto F_{\beta}(U,\tU)$ defines a sheaf on $(X,\tX)_t$. 
\end{prop}

\begin{para}\label{para:Ot}
Taking $F=\sO_X$ and $\beta=\{F_w=\sO^h_{\tX,L,w}\}$, where $\sO^h_{\tX,L,w}:=\sO^h_{X,L}\times_L \sO_w$, 
we obtain the tame structure sheaf on $(X,\tX)_t$ 
denoted by $\sO^t$. By \cite[Lemma 8.5]{MRS-1} 
\eq{eq:Ot}{\sO^t(U,\tU)=\sO(\tU^{\rm int}),\quad \text{for all } (U,\tU)\in (X,\tX)_t,}
where $\tU^{\rm int}$ denotes the integral closure of $\tU$ in $U$.
\end{para}

\section{A comparison theorem with the cohomology of rigid analytic spaces}

Let $(A,I)$ be a universally adhesive pair in the sense of \cite[Definition \FKC{0}.8.5.4]{FK} with $A$ universally coherent 
\cite[Definition \FKC{0}.3.3.7]{FK}. For example, $A=\cO_K$ is the ring of integers of a complete non-archimedean field $K$ and $I=(a)$, where $a$ is a pseudo-uniformizer, or more generally if $A$ is an $a$-adically complete valuation ring and $I=(a)$ 
\cite[Corollaries \FKC{0}.9.2.7 and \FKC{0}.9.2.8]{FK}. In this section, we fix 
\begin{itemize}
    \item $\tS=\Spec(A)$
    \item $S=\Spec(A)\setminus V(I)$
    \item $\tX$ of finite presentation over $\tS$
    \item $X = \tX\times_{\tS} S$.
\end{itemize}
\begin{defn}\label{def;AdBl}
Let $U\subseteq X$ be a quasi-compact dense open immersion, and let $\cJ$ be a finitely generated quasi-coherent ideal in $\cO_{\tX}$
with vanishing locus the complement $\tX-U$. 
Since $\tX$ is of finite presentation and $A$ is universally coherent, $\cJ$ is of finite presentation. 
Let $\StX(U)$ be the category of $U$-admissible blowups of $\tX'\to \tX$, i.e. blowups 
$\rho_{\cA}: \Bl_{\cA}(\tX) \to \tX$ in a $U$-admissible ideal $\cA\subset\cO_{\tX}$, i.e. a quasi-coherent ideal sheaf of finite presentation $\cA\subset\cO_{\tX}$ containing a power of $\cJ$.
The morphisms in $\StX(U)$ are morphisms $h:\Bl_{\cA'}(\tX)\to \Bl_{\cA}(\tX)$, where $\cA'=\cA\cA''$ for some admissible ideal $\cA''\subset \cO_{\tX}$ and $h$ are induced by the universality of blowups. By construction, $\StX(U)$ is a cofiltered category. If $U=X$, we simply write $\StX$.

\end{defn}

\begin{para}\label{para:tame-qcoh}
Denote by  $\Sh((X,\tX)_t,\cO^t)$ the category of sheaves of $\cO^t$-modules, with $\cO^t$ as in \eqref{eq:Ot} and
by $\ShvQCoh{(X,\tX)}$ (resp. $\ShvCoh{(X,\tX)}$)  the subcategory  of sheaves $F$ such that for all 
$(U,\tY)\in (X,\tX)_t$ the sheaf $F_{(U,\tY)_{\et}}$ is a quasi-coherent (resp. coherent) $\cO_{\tY_{\et}}$-module, cf. Definition \ref{def;tametopology-intro}\ref{def;tametopology-intro5}.  
If $F\in \ShvQCoh{(X,\tX)}$, then for every \'etale morphism $e\colon \tY'\to \tY$, we have a canonical isomorphism\[
F(Y\times_{\tY}\tY',\tY')\cong (e^*F_{(Y,\tY)_\et})(\tY'),
\]
see  \cite[\href{https://stacks.math.columbia.edu/tag/03OJ}{Tag 03OJ}]{stacks-project}.
Thus by Theorem \ref{thm:tame-vs-etale-intro} there is a canonical isomorphism
\eq{eq:et-zar}{
	H^i((U,\tX)_t, F)\cong \colim_{\tY \in \StX(U)} H^i(\tY, F_{(U,\tY)_{\Zar}}), \quad i\ge 0.
	}

Let $f\colon \tY'\to \tY\in \StX(U)$. As $F\in \ShvQCoh{(X,\tX)}$ 
is invariant under quasi-modifications, see \cite[Lemma 2.7]{MRS-1}, we have an isomorphism 
$F_{(U,\tY)_{\et}}\xrightarrow{\simeq} f_*F_{(U,\tY')_{\et}}$, which is a morphism of $\cO_{\tY_{\et}}$-modules.
By adjunction we obtain induced morphisms of $\sO_{\tY'_{\et}}$-modules.
\eq{para:tame-qcoh1}{f^*: f^*F_{(U,\tY)_{\et}}\to F_{(U,\tY')_{\et}}.}
\end{para}

\begin{rmk}
If $A$ is not a Nagata ring, then $\cO^t_{(U,\tY)_{\Zar}}$ is in general not an $\cO_{\tY}$-module of 
finite presentation if $\tY \in \LtX(U)$.
\end{rmk}

\begin{para}\label{para:formal-completion}
Let $\fX$ be the formal completion of $\tX$ along $V(I)$ and let 
\eq{para:formal-completion1}{c^*\colon \mathbf{Mod}_{\tX}\to \mathbf{Mod}_{\fX}}
be the pullback of modules along the natural morphism  of locally ringed spaces $c\colon \fX\to \tX$.
Under the running assumption of this section the map $c$ is  flat and hence $c^*$ is exact, 
see \cite[Proposition \FKC{I}.1.4.7(2)]{FK}.
Following \cite[\FKC{I}.9.1(a)]{FK} we define the formal completion of the 
$\sO_{\tX}$-module $F_{(X,\tX)_{\Zar}}$ associated to $F\in \ShvQCoh{(X,\tX)}$, by
\[
F_{(X,\tX)}^{\rm for}:= c^* F_{(X,\tX)_{\Zar}}\in \mathbf{\Mod_{\fX}}. 
\]
If $F_{(X,\tX)_{\Zar}}$ is a finitely generated quasi-coherent sheaf, 
then we have an identification 
\[
F_{(X,\tX)_{\Zar}}^{\rm for}=\lim_n F_{(X,\tX)_{\Zar}}/I^n F_{(X,\tX)_{\Zar}},\]
where the right term is the usual completion, see \cite[Proposition \FKC{I}.1.4.7(1)]{FK}.
Let $f\colon \tY\to \tX$ be a proper finitely presented morphism of $\tS$-schemes of finite type and $Y= \tY\times_{\tS} S$. Let $\hat f:\fY\to \fX$ be the formal completion of $f$ along $V(I)$. 
By GFGA \cite[Theorem \FKC{I}.9.1.3]{FK}, for all $F\in \ShvCoh{(Y,\tY)}$ we have 
\eq{eq:GFGA}{
c^*Rf_*(F_{(Y,\tY)_{\Zar}}) \simeq R\hat f_*(F_{(Y,\tY)}^{\rm for}).
}
\end{para}
\begin{para}\label{para:RZ}
Let $\fX$ be the formal completion of $\tX$ along $V(I)$. 
Recall that an admissible ideal of $\fX$ in the sense of \cite[Definition \FKC{I}.3.7.4]{FK}
corresponds to an ideal of finite type in $\cO_{\fX}/I^n=\cO_{\tX}/I^n$, for some $n\ge 1$, by \cite[Corollary \FKC{I}.37.3]{FK}.
Hence an admissible ideal on $\fX$ is the same as an ideal in $\cO_{\tX}$ containing a power of $I$. 
Therefore, an admissible blowup of $\fX$ is by its definition, e.g. \cite[Definition \FKC{II}.1.1.1]{FK},
the same as the completion along $V(I)$ of a blowup of $\tX$ in an ideal containing a power of $I$.
Thus we can identify the Riemann--Zariski space $\RZ(\fX)$ of $\fX$ as defined in  \cite[\FKC{II}.3.1]{FK}
with the locally ringed space defined as the limit (as locally ringed space) of $(\fX',\cO_{\fX'})$ for $\tX'\in \StX$, where $\fX'$ is the formal completion of $\tX'$ along $V(I)$, that is as topological space
\[\RZ(\fX)=\lim_{\tX'\in \Sigma_{\tX}} \fX',\]
with (integral) structure sheaf given by
\[\cO_{\RZ(\fX)}^o= \colim_{\tX'\in \Sigma_{\tX}}p_{\fX'}^{-1}\cO_{\fX'},\]
where {$p_{\fX'}\colon \RZ(\fX)\to \fX'_{\zar}$} is the projection from the limit. 
The quasi-compact open subsets of $\RZ(\fX)$ form a basis of the topology and are of the form
$p_{\fX'}^{-1}(\fU')$, where $\fU'$ is the completion of a quasi-compact open $\tU'\subset \tX'$ along $\tU'\cap V(I)$
If $A= \cO_K$ with $K$ a complete non-archimedean field, then 
there is a canonical equivalence of ringed topoi  (e.g.  \cite[\FKC{II}.B.2.(e)]{FK})
\eq{eq2;FRZ}{ 
(\Sh(\fX^{\rig}), \sO_{\fXrig}^o)\simeq (\Sh(\RZ(\fX)), \cO^o_{\RZ(\fX)}),
}
where $\fXrig$ is the (classical) rigid analytic space associated to the formal model $\fX$ (e.g. \cite[7.4, Proposition 3]{Bosch-RigGeom})
and the sheaf of rings $\Oc_{\fXrig}$ on $\fXrig$ is given by 
\[\fXrig\supset \Sp(B)\mapsto \cO^o_{\fXrig}(\Sp(B))=\{f\in B|\; |f|_{\sup}\leq 1 \}\; ,\]
where $|-|_{\sup}$ is the supremum norm on the affinoid $K$-algebra $B$.
\end{para}

The following lemma will be used in section \ref{sec:4}.
\begin{lemma}\label{lem:Orig-flat-Ot}
Assume $A=\sO_K$ with $K$ a complete non-archimedean field.
Let $\tU=\Spec\tB\subset \tX$ be open and denote by $\fU$ the completion of $\tU$ along $V(I)$.
Set $U=\tU\times_{\tX} X$ and denote by $\tU^{\rm int}$ the integral closure of $\tU$ in $U$. 
Then the canonical  morphism  
$\cO(\tU^{\rm int})\to \cO^o_{\fXrig}(\fU^{\rig})$ is flat.

Moreover, in case the integral closure $\tX^{\rm int}$ of $\tX$ in $X$ is finite over $\tX$, the natural map
\eq{lem:0rig-flat-Ot0}{\cO_{\tilde{\fX}^{\int}}(\tilde{\fU}^{\int})\xrightarrow{\simeq} \cO^o_{\fXrig}(\fU^{\rig})}
is an isomorphism, where $\tilde{\fX}^{\rm int}$ is the completion of $\tX^{\rm int}$ along $V(I)$ 
and similarly with $\tilde{\fU}^{\int}$.
\end{lemma}
\begin{proof}
 By \eqref{eq:Ot} and Theorem \ref{thm:tame-vs-etale-intro},
\[\sO(\tU^{\rm int})=\sO^t(U, \tU)=\colim_{\tX'\in \Sigma_{\tX}} \sO_{\tX'}(\tU')\]
By \eqref{eq2;FRZ}, 
\[\colim_{\tX'\in \Sigma_{\tX}}p_{\fX'}^{-1} (\lim_n \sO_{\tX'}/I^n\sO_{\tX'})=\cO^o_{\fXrig},\]
where $p_{\fX'}^{-1}:\Sh(\fX'')\to \Sh(\RZ(\fX))\cong \Sh(\fXrig)$ is induced by the projection map.
By the Formal Function Theorem  \cite[Theorem \FKC{I}.9.2.1]{FK} we have for $\tX'\in \StX$
\eq{lem:Orig-flat-Ot1}{ \lim_n B_{\tU'}/I^n B_{\tU'} = H^0(\fU', \lim_n \sO_{\tX'}/I^n\sO_{\tX'})=
H^0(p_{\fX}^{-1}(\fU), p_{\fX'}^{-1}(\lim_n \sO_{\tX'}/I^n\sO_{\tX'})),}
where $\tU'=\tX'\times_{\tX} \tU$, $\fU'$ is its completion along $V(I)$,  and $B_{\tU'}= \sO_{\tX'}(\tU')$. 
By \cite[Theorem \FKC{I}.8.2.1]{FK}, $B_{\tU'}$ is a finite $\tB$-algebra, hence  $(B_{\tU'}, IB_{\tU'})$ is universally adhesive.
Thus, by \cite[Lemma \FKC{0}.8.2.18(2)]{FK}, the natural maps
\[B_{\tU'}\to \lim_n B_{\tU'}/I^n B_{\tU'} \]
are flat. Hence its colimit 
\[\sO(\tU^{\rm int})=\colim_{\tX\in \StX} B_{\tU'}\to \colim_{\tX'\in \StX}\left(\lim_n B_{\tU'}/I^n B_{\tU'}\right)\]
is flat as well, e.g. \cite[\href{https://stacks.math.columbia.edu/tag/05UU}{Tag 05UU}]{stacks-project}.
As $\RZ(\fX)$ is quasi-compact quasi-separated, see \cite[Theorem \FKC{II}.3.1.2]{FK}, we have
by \eqref{lem:Orig-flat-Ot1}, \eqref{eq2;FRZ}, and \cite[Corollary \FKC{0}.3.1.9]{FK}
\[\colim_{\tX'\in \StX}\left(\lim_n B_{\tU'}/I^n B_{\tU'}\right)
= H^0(p_{\fX}^{-1}(\fU), \colim_{\tX'\in \StX} p_{\fX'}^{-1}(\lim_n \sO_{\tX'}/I^n\sO_{\tX'}))
=\cO^o_{\fXrig}(\fU^{\rig}).\]
This completes the proof of the first part. 

Now assume that $\tX^{\int}$ is finite over $\tX$. In this case $\Sigma_{\tX^{\int}}$ is cofinal in $\Sigma_{\tX}$.
Moreover, for any $\tX'\in \Sigma_{\tX^{\int}}$ and $\tU'$ as above 
we have that $B_{\tU'}= \cO_{\tX^{\int}}(\tU^{\int})$, where $\tU^\int$ is in this case also equal to $\tX^{\int}\times_{\tX} \tU$. 
Thus the last displayed formula above yields 
\[\lim_n \cO_{\tX^{\int}}(\tU^{\int})/I^n = \cO^o_{\cX^{\rig}}(\fU^{\rig}),\] 
which yields the second statement.
\end{proof}

\begin{para}\label{para:tame-rig-sheaf}
We  associate to $F\in \ShvQCoh{(X,\tX)}$ a sheaf on $\RZ(\fX)$ as follows: 
 For all $\tX'\in \StX$ consider the formal completion $\fX'$ of $\tX'$ along $V(I)$ and the projection 
 $p_{\fX'}:\RZ(\fX)\to \fX'_{\zar}$. We get a filtered direct system on $\RZ(\fX')$
\eq{eq;FRZ}{\{p_{\fX'}^*F_{(X,\tX')}^{\rm for}\}_{\tX'\in \Sigma_{\tX},}
}
{where a morphism  $g\colon \tX''\to \tX'$  in $\StX$ gives rise to a commutative diagram
\[\xymatrix{
\RZ(\fX)\ar[r]^-{p_{\fX''}}\ar[dr]_{p_{\fX'}} & \fX''\ar[d]^{\hat{g}}\ar[r]^{c_{\tX''}}  & \tX''\ar[d]^g\\ 
                                             & \fX'\ar[r]^{c_{\tX'}}                   &\tX',
}\]
inducing the transition map 
\[
p_{\fX'}^*F_{(X,\tX')}^{\rm for}=
p_{\fX''}^*c_{\tX''}^*g^*F_{(X,\tX')}\xrightarrow{g^*} p_{\fX''}^*c_{\tX''}^*F_{(X,\tX'')}=p_{\fX''}^*F_{(X,\tX'')}^{\rm for}. 
\]
We define 
\eq{eq:rig1}{ F_{\RZ(\fX)}:=\colim_{\tX'\in \Sigma_{\tX}} p_{\fX'}^{*} F_{(X,\tX')}^{\rm for}\in \Shv(\RZ(\fX),\Oc_{\RZ(\fX)}).}}
In case $A=\sO_K$ with $K$ a complete non-archimedean field we obtain by \eqref{eq2;FRZ} 
a corresponding sheaf of $\cO^o_{\fXrig}$-modules on the rigid analytic space $\fXrig$ denoted by 
\eq{eq:rig2}{F_{\fXrig}\in \Shv(\fXrig,\cO^o_{\fXrig}).}

\begin{lemma}\label{lem:sectionsFrig}
Assume $A=\sO_K$ with $K$ a complete non-archimedean field.
Let $F\in \ShvQCoh{(X,\tX)}$.
Then $F_{\fXrig}$ is the sheaf associated to the presheaf 
\[ \fU^{\rig}\mapsto  F(U, \tU)\otimes_{\cO(\tU^{\rm int})} \cO^o_{\fX^{\rig}}(\fU^{\rig}),\]
defined on the basis of the topology consisting of 
all admissible opens of the form  $\fU^{\rig}$ with $\fU$ the completion along $V(I)$ of an affine open $\tU\subset \tX'$
with $\tX'\in \Sigma_{\tX}$, and where $U:=\tU\times_{\tS} S$ and $\tU^{\rm int}$ denotes the integral closure of $\tU$ in $U$.
\end{lemma}
\begin{proof}
By \cite[Lemma \FKC{0}.4.2.7]{FK} we have the following equality in $\Shv(\RZ(\fX),\cO_{\RZ(\fX)})$
\[\colim_{\tX'\in \Sigma_{\tX}} p_{\fX'}^{*} F_{(X,\tX')}^{\rm for}=
\colim_{\tX'\in \Sigma_{\tX}} p_{\fX'}^{-1} F_{(X,\tX')}^{\rm for}.\]
Therefore, $F_{\RZ(\fX)}$ is the sheaf associated to the  presheaf on the basis of the topology
$\{\RZ(\fU)\}_{\tU\subset \tX'}$, where  $\tU$ runs over all affine open subsets of $\tX'\in \StX$ and $\fU$ 
is its completion along $V(I)$,
\begin{align*}
\RZ(\fU)\mapsto &  
\colim_{\tX''\in \Sigma_{\tX'}} 
p_{\fX''}^{-1} \left(c_{\tX''}^{-1}F_{(X,\tX'')}\otimes_{c_{\tX''}^{-1}\cO_{\tX''}}\cO_{\fX''}\right)(p^{-1}_{\fX'}(\fU))      \\
=& \colim_{\tX''\in \Sigma_{\tX'}} \left(
F(\tU''\times_{\tX''} X, \tU'')\otimes_{\sO(\tU'')} (\lim_n\sO(\tU'')/I^n\sO(\tU''))\right),
\end{align*}
where $\tU''=\tX''\times_{\tX'} \tU$ and the inverse images and tensor products above are formed in a presheaf sense.
As $\tU''\to \tU$ is a modification we have an isomorphism 
\[F(\tU''\times_{\tX''} X, \tU'')\cong F(U, \tU)\] 
of $\sO^t(U, \tU)$-modules, where $U=\tU\times_{\tX'} X$.
By \eqref{eq:Ot} and Theorem \ref{thm:tame-vs-etale-intro},
\[\colim_{\tX''\in \Sigma_{\tX'}} \sO(\tU'')=\sO^t(U, \tU)=\sO(\tU^{\rm int}).\]
Moreover, by \eqref{eq2;FRZ} and the fact that composition with $\tX'\to\tX$ induces a cofinal functor $\Sigma_{\tX'}\to \StX$
we get 
\[\colim_{\tX''\in \Sigma_{\tX'}}p_{\fX''}^{-1} (\lim_n \sO_{\tX''}/I^n\sO_{\tX''})=\cO^o_{\fXrig},\]
where $p_{\fX''}^{-1}:\Sh(\fX'')\to \Sh(\RZ(\fX))\cong \Sh(\fXrig)$ is induced by the projection map.
In view of  \eqref{eq2;FRZ} this yields the statement.
\end{proof}

\end{para}

\begin{thm}\label{thm;comparison}
Assume $A$ is $I$-adically complete and $\tX$ is a proper $\tS$-scheme of finite presentation. 
Denote by $\fX$ the completion of $\tX$ along $V(I)$. 
For $F\in \ShvCoh{(X,\tX)}$, we have
\[
R\Gamma((X,\tX)_t,F)\simeq R\Gamma(\RZ(\fX),F_{\RZ(\fX)}). 
\]
If  $\tS=\Spec \sO_K$ with $K$ a complete non-archimedean field, then
\[
R\Gamma((X,\tX)_t,F)\simeq R\Gamma(\fXrig,F_{\fXrig}). 
\]
\end{thm}
 \begin{proof}
By \eqref{eq2;FRZ} the second statement follows directly from the first. 
The first statement holds by the following sequence of isomorphisms
 \[\begin{aligned}
 R\Gamma((X,\tX)_t,F)& \overset{(*1)}{\simeq}
 \colim_{\tY\in \StX} R\Gamma(\tY,F_{(X,\tY)_{\Zar}})\\
 &
 \overset{(*2)}{\simeq} \colim_{\tY\in \SX} R\Gamma(\fY,F_{(X,\tY)}^{\rm for}) 
 \overset{(*3)}{\simeq} R\Gamma(\RZ(\fX),F_{\RZ(\fX)}),\\
 \end{aligned}\]
where $(*1)$ follows from \eqref{eq:et-zar}, $(*2)$ from \eqref{eq:GFGA} and the fact that $A$ is $I$-adically complete, 
and $(*3)$ from \cite[Proposition \FKC{0}.4.4.1]{FK}.
\end{proof}

\begin{defn}\label{defn:tb}
    We say that $F\in \Shv{((X,\tX)_t)}$ is \emph{tamely birational} if for every object $(X',\tX')\in (X,\tX)_t$ and all $U\subseteq X'$ open dense, the restriction map induces an isomorphism $F(X',\tX')\cong F(U,\tX')$. In particular, the Zariski sheaves $F_{(U,\tX')_{\Zar}}$ and $F_{(X',\tX')_{\rm Zar}}$ coincide on $\tX'_{\Zar}$. 
\end{defn}
\begin{ex}
   When $X$ is normal, the sheaves $\cO^t$ and $\Gar$ of Definition \ref{def;Gar} are tamely birational, by \eqref{eq:Ot} and Lemma \ref{lem:Gar-shift}\ref{lem:Gar-shift1}, respectively.
\end{ex}

\begin{thm}\label{thm:colim-smooth-hironaka}
     Let $A$ be integral and set $K=\Frac(A)$. Assume $A$ is $I$-adically complete,  
    $S = \Spec(K)$\footnote{For example, if $I=(a)$ and $K=A[1/a]$.},  
  and $\tX\to \Spec A$ is proper of finite presentation. 
    Assume either $\ch(K)=0$ or $\ch(K)=p>0$ and $\dim X\le 3$. 
    Let $U\subseteq X$ be a dense open such that $X$ is regular in a neighborhood of $X-U$. 
    Let $F\in \ShvCoh{(X,\tX)}$ be tamely birational. Then we have\[
    R\Gamma((U,\tX)_t,F)\simeq \varinjlim_{\tX'\in\StX'(U)} R\Gamma(\RZ(\fX'),F_{\RZ(\fX')}),
    \]
    where $\StX'(U)\subseteq \StX$ is the subcategory of $U$-admissible blowups $\tX'\to \tX$ in a center $\tT$ such that $T:=\tT\times_{\tS} S$ is regular.
\end{thm}
\begin{proof}
By \eqref{eq:et-zar}, \[
R\Gamma((U,\tX)_t,F) \simeq \colim_{\tZ\in \StX(U)} R\Gamma(\tZ,F_{(U,\tZ)_{\Zar}}).
\]
We can rewrite the right hand side as \[
\colim_{\tX'\to \tX}\colim_{\tY\in\St(\tX')}  R\Gamma(\tY,F_{(U,\tY)_{\Zar}}),
\]
where $\tX' \to \tX$ ranges over the blowups in centers contained in the closure of $X\setminus U$ in $\tX$.
Indeed, if $\tZ\to \tX$ is an $U$-admissible blowup in the ideal sheaf $\cI\subset\cO_{\tX}$, then 
we get a canonical  $\tX$-morphism $\tY:={\rm Bl}_{\cI \cJ}(\tX)\to \tZ$, where 
$\cJ=\Ker(\cO_{\tX}\to j_*\cO_X\to j_*\cO_X/\cI_{|X})$ is the ideal of the scheme-theoretic closure of $V(\cI_{|X})$, 
and $\tY\to\tX$ is the composition 
$\tY\xrightarrow{q}\tX'={\rm Bl}_{\cJ}\tX\to\tX$, where $q$ is the blowup in $\cI\cO_{\tX'}$, which over
$X$ is invertible and hence $(\tY'\to\tX')\in \Sigma_{\tX'}$.

Note $F_{(U,\tY)_{\Zar}} = F_{(X',\tY)_\Zar}$ by the hypothesis that $F$ is tamely birational,
where $X':= \tX'\times_{\tX} X$. 
Hence, combining this with Theorem \ref{thm;comparison}, we get an equivalence
\[
R\Gamma((U,\tX)_t,F) \simeq \varinjlim_{\tX'\to \tX} R\Gamma(\RZ(\fX'),F_{\RZ(\fX')}).
\]
Fix $\tX'=\bu_{\cJ}(\tX)$, where $\cJ\subset \cO_{\tX}$ is 
a finitely presented quasi-coherent sheaf of ideals by the assumption that $\tX$ is of finite presentation over the  universally coherent ring $A$. 
Since $X$ is regular in a neighborhood of $X\setminus U$, there exists a sequence of $U$-admissible blowups $X_r\to  \cdots \to X_1\to X$ 
whose centers $C_i\subset X_i$ are regular such that $\cJ\cO_{X_r}$ is invertible.
In case ${\rm char}(K)=0$, this holds by  Hironaka's principalization, \cite[Corollary 1]{Hironaka};
in case ${\rm char}(K)=p>0$ and $\dim X\le 3$, this holds by \cite[Proposition 4.1]{CossartPiltant-ResI}.

We extend the sequence to $\tX_r\to \cdots \to \tX_1\to \tX$,
where $\tX_{i+1}\to \tX_i$ is the blowup of $\tX_i$ in the closure of $C_i$.
Let $\tX_{r+1}\to \tX_{r}$ be the blowup in $\cJ\cO_{\tX_r}$; it is an isomorphism over $X_r$.
Thus $\tX_{r+1}\to\tX_{r}\to \tX\in \Sigma'_{\tX}(U)$.
Moreover, as $\cJ\cO_{\tX_{r+1}}$ is invertible  the composition $\tX_{r+1}\to \tX$ factors via $\tX' \to\tX$,
by the universal property of the blowup $\tX'={\rm Bl}_{\cJ}(\tX)$. Hence 
$\Sigma'_{\tX}(U)$ is cofinal in the category of all blowups of $\tX$ with center contained in the closure of $X\setminus U$ in $\tX$.
\end{proof}

Using a notion of almost coherence introduced by Bogdan Zavyalov in \cite{Zavyalov} 
we can relax in the arguments above the condition that $F_{(X,\tY)}$ 
is a coherent $\sO_{\tY_{\et}}$-module in certain situations considerably. 
This will be useful in section \ref{sec:4} when dealing with algebraically closed non-archimedean fields.
Here is the general setup.

\begin{para}\label{para:ac}
Let $(R,I,\fm)$ be a triple satisfying the hypotheses of \cite[Set-up 4.5.1]{Zavyalov}, i.e., $I\subseteq \fm$ are ideals of a commutative ring $R$ such that
\begin{itemize}
    \item $I$ is finitely generated, 
    \item $R$ is $I$-adically complete, $I$-torsionfree and $I$-adically topologically universally adhesive,
    \item $\fm^2 = \fm$ and $\fm\otimes_R \fm$ is $R$-flat.
\end{itemize}
For example, $R=\cO_K$ is the ring of integers of an algebraically closed field complete with respect to a non-archimedean absolute value, $I=(a)$, and $\fm = \cup_{m >0} (a^{1/m})$, for some pseudo-uniformizer $a\in R$. 

In this setting, we let $A$ be a topologically finitely presented $R$-algebra, and as before we fix:
\begin{itemize}
    \item $\tS=\Spec(A)$,
    \item $S=\Spec(A)\setminus V(IA)$,
    \item $\tX$ of finite presentation over $\tS$,
    \item $X = \tX\times_{\tS} S$.
\end{itemize}
We consider the category $\ShvaCoh{(X,\tX)}$ of sheaves $F$ on $(X,\tX)_t$
such that for all $(U,\tY)\in (X,\tX)_t$ the sheaf $F_{(U,\tY)_{\et}}$ is a quasi-coherent $\cO_{\tY_{\et}}$-module 
and its restriction to the Zariski site $F_{(U,\tY)_{\Zar}}$ is almost coherent in the sense of \cite[Definition 4.1.10]{Zavyalov}. 
By the formal function theorem  for quasi-coherent almost coherent sheaves \cite[Theorem 5.4.7]{Zavyalov}, 
we have an equivalence
\eq{para:ac1}{
R\Gamma(\tX, F_{(X,\tX)_{\Zar}})\otimes_A \hat{A} \simeq R\Gamma(\fX, F_{(X,\tX)}^{\rm for}), \quad F\in \ShvaCoh{(X,\tX)},
}
where $\hat{A}$ denotes the $I$-adic-completion  of $A$, which is a flat $A$-module, by \cite[Proposition \FKC{0}.8.2.18]{FK}.
Using this equivalence instead of \eqref{eq:GFGA} the following theorem is proven analogously to Theorem \ref{thm;comparison}.
\begin{thm}\label{thm:almost-comparison}
Assume $A$ is $I$-adically complete and $\tX$ is a proper $\tS$-scheme of finite presentation.
Denote by $\fX$ the completion of $\tX$ along $V(I)$. For $F\in \ShvaCoh{(X,\tX)}$ we have
\[
R\Gamma((X,\tX)_t,F)\simeq R\Gamma(\RZ(\fX), F_{\RZ(\fX)}).
\]
If $\tS=\Spec \cO_K$ with $K$ an algebraically closed complete non-archimedean field, then
\[R\Gamma((X,\tX)_t, F)\simeq R\Gamma(\fXrig, F_{\fXrig}).\]
\end{thm}

Therefore, the same proof of Theorem \ref{thm:colim-smooth-hironaka} gives:
\begin{thm}\label{thm:almost-colim-smooth-hironaka}
 Let $A$ be integral and set $K=\Frac(A)$. Assume $A$ is $I$-adically complete,   $S = \Spec(K)$,
 and $\tX\to \Spec A$ is proper of finite presentation. 
   Assume either $\ch(K)=0$ or $\ch(K)=p>0$ and $\dim X\le 3$.
    Let $F\in \ShvaCoh{(X,\tX)}$ be tamely birational.  
    Let $U\subseteq X$ be a dense open such that $X$ is regular in a neighborhood of $X-U$.  Then we have
    an equivalence\[
    R\Gamma((U,\tX)_t,F)\simeq \varinjlim_{\tX'\in\StX'(U)} R\Gamma(\RZ(\fX'),F_{\RZ(\fX')}),
    \]
    with $\Sigma'_{\tX}(U)$ as in Theorem \ref{thm:colim-smooth-hironaka}.
\end{thm}
\end{para}

\section{Cohomology of affinoids spaces}
In this section, we collect some results on cohomology of affinoid spaces.
Let $K$ be a field complete with respect to the non-archimedian absolute value $|-|_K:K \to \R_{\ge 0}$, which is assume to be non-trivial. 
We put
\[R:=\cO_K=\{x\in K|\; |x|\leq 1\}.\]

\begin{defn}\label{defn:Or}
	For a rigid space $X$ over $K$ and $r\in \R_{>0}$, we have a sheaf $\Or$ on $X$ given by 
	\eq{Or}{\Or(B)=\{f\in B|\; |f|_{\sup}<r\}\; \text{for an affinoid subdomain }\Sp(B)\subset X,}
	where $|-|_{\sup}$ is the sup norm on $B$.
	We have also a sheaf $\Oc$ on $X$ given by 
	\eq{Oc}{\Oc(B)=\{f\in B|\; |f|_{\sup}\leq {1}\}\; \text{for an affinoid subdomain }\Sp(B)\subset X.}
\end{defn}
\medbreak

The following theorem was first shown by Bartenwerfer in
\cite[Folgerung~3]{B2} (see also \cite[Theorem 17]{KST-BQ}).

\begin{thm}\label{thm.bartenwerfer}
	Let $X/K$ be a smooth affinoid.	
    		For all $i>0$, there exist $c\in K^\times$ with $|c|<1$ such that $c\cdot H^i(X,\Or)=0$, for all $r\in \R_{>0}$.
	\end{thm}
\begin{rmk}
    By the vanishing theorem of \cite[Lemma (1.4.13)]{vdP}, we can choose $c$ in Theorem \ref{thm.bartenwerfer} such that $c \cdot H^i(X,\Oc)=c \cdot H^i(X,\Or)=0$ for all $i>0$ and all $r\in \R_{>0}$. 
\end{rmk}
\begin{para}\label{para:gen-poly-disc}
Recall from \cite[(3.9)]{vdP} that a {\em generalized polydisk} over $K$ is a subset $D\subset \bB^d=\Sp(K\langle T_1,\dots,T_d\rangle)$
of the form 
\[D=D_1\times\ldots\times D_n,\]
where each $D_k\subset \bB^1$ is given by 
\[D_k =\{z\in \bB^1\mid  |z-a|\le \rho \quad \text{and}\quad |z-a_i|\ge \rho_i, \quad i=1,\ldots,s \},\]
where $a, a_1,\ldots, a_s\in K$ and $\rho,\rho_1,\ldots, \rho_s\in |K^\times|$, such that 
$|a|\le 1$, $\rho\le 1$, $|a-a_i|\le \rho$, $\rho_i\le \rho$, $|a_i-a_j|\ge \rho_i$, for $i\neq j$, 
and $s\ge 0$. We are mostly interested in generalized polydisks of the form
\[\Sp(K\langle T_1\rangle)\times\ldots \times \Sp(K\langle T_j\rangle)\times 
\Sp(K\langle T_{j+1}, T_{j+1}^{-1}\rangle)\times \ldots\times \Sp(K\langle T_{d}, T_{d}^{-1}\rangle).\]
\end{para}

We will also use the following theorem due to Bartenwerfer \cite[Theorem]{B2} and 
van der Put \cite[(3.15) Theorem and (3.20) Corollary]{vdP}.

\begin{thm}\label{thm.van1}
	For a generalized polydisk $D\subset \bB^d$, we have for all $r>0$ 
	\[H^i(D,\sO(r))=0,\quad\text{for }i>0, \quad \text{and}\quad H^0(D,\sO(r))=K(r)\hat{\otimes}_R H^0(D,\cO^o),\]
    where $K(r)=\cO(r)(\Sp(K))=\{a\in K\mid |a|<r\}$.
\end{thm}

We deduce from the above theorem a necessary result for this note.

\def\hC{\hat{C}}
\def\hM{\hat{M}}
\def\Uring{U^{\ring}}

\begin{lemma}\label{lem1;CohRigid}
	Let  $\P^n$ be the rigid projective space over $K$. 
    Then for all $s> 0$
	\[R\Gamma(\P^n, \cO(s))=K(s)[0],\]
    where $K(s)[0]$ is the $R$-module $\cO(s)(\Sp (K))=\{a\in K\mid |a|_{\rm sup}<s\}$ sitting in degree $0$.
\end{lemma}
\begin{proof}
	
	Consider the standard affine open covering
	\[ \Proj R[T_0,\dots,T_n] =\underset{0\leq i\leq n}{\bigcup} \Spec C_i\;
	\text{ with } C_i = R[T_0,\dots,T_n,\frac{1}{T_i}]_0.\]
	Taking the $\pi$-adic completions, it induces a covering 
	\[ \P^n=\underset{0\leq i\leq n}{\bigcup} \Sp(\hC_i[1/\pi]),\]
	where $\hC_i$ is the $\pi$-adic completion of $C_i$. 
	For $1\leq i_0<\cdots<i_p\leq n$, set 
	\[D_{i_0,\ldots, i_p}:= \Sp(\hat{C}_{i_0}[1/\pi])\times_{\P^n} \cdots\times_{\P^n} \Sp(\hat{C}_{i_p}[1/\pi]).\]
    We have 
    \[D_{i_0,\ldots, i_p}=\Sp(\hC_{i_0,\dots,i_p}[1/\pi]), \quad \text{where }
    C_{i_0,\dots,i_p}=R[T_0,\dots,T_n,\frac{1}{T_{i_0}\cdots T_{i_p}}]_0.\]
	By Theorem \ref{thm.van1} we have
	\[H^i(D_{i_0,\dots,i_p},\cO(s))=0, \quad  \text{for } i>0, \quad \text{and}\quad 
    H^0(D_{i_0,\dots,i_p},\cO(s))= K(s)\hat{\otimes}_R \hC_{i_0,\ldots, i_p}.\] 
    Hence, $R\Gamma(\P^n,\cO(s))$ is computed by using the \v Cech complex $\hM^\bullet$ with  
	\[ \hM^p:=\underset{ i_0<\cdots<i_p}{\bigoplus} K(s)\hat{\otimes}_R \hC_{i_0,\dots,i_p}.\]
    Consider the complex $K(s)^\bullet$ with $K(s)^p=\oplus_{i_0<\ldots<i_p} K(s)$ and  differentials $K(s)^p\to K(s)^{p+1}$
    given by the usual alternating sum. Then the natural augmentation map $K(s)[0]\to K(s)^\bullet$ is a quasi-isomorphism and it
    suffice to show that $\hM^\bullet_{\red}:={\rm coker}(K(s)^\bullet\to \hM^\bullet)$ is acyclic. 
    Now, consider the \v Cech complex of $R$-modules $M^\bullet$ with  
	\[ M^p:=\underset{ i_0<\cdots<i_p}{\bigoplus} K(s) \otimes_R C_{i_0.\dots,i_p},\]
	and the reduced complex $M^\bullet_{\red}$ with $M^p_{\red} = M^p/K(s)$.
	By the proof of \cite[\href{https://stacks.math.columbia.edu/tag/01XT}{Tag 01XT}]{stacks-project}, there exists a homotopy
	$h: M^{p+1}_{\red}\to M^p_{\red}$ such that $dh +hd=\id$, where $d: M_{\red}^\bullet\to M_{\red}^{\bullet+1}$ is the differential. It is obviously continuous with respect to the $\pi$-adic topology so that it induces a homotopy $\hat{h}: \hM^{p+1}_{\red}\to \hM^p_{\red}$ on the $\pi$-adic completions, which implies the desired acyclicity of $\hM^\bullet_{\red}$.
\end{proof}

\section{Proofs of the main theorems}\label{sec:4}

\begin{para}\label{para:setup}
We continue to denote by $|-|: K\to \R_{\ge 0}$ the absolute value on the complete non-archimedean field $K$ and denote by 
$R=\sO_K$ its ring of integers. We fix $\pi\in R\setminus\{0\}$ with $|\pi|:=c\in (0,1)$ and denote by
\[v_K: K^\times\to \R, \quad v_K(x):=\log_c|x| \]
the associated valuation. Note that $v_K(\pi)=1$. 

Throughout this section we assume that the valuation ring $R$ is N-2 and 
$\tX$ is a faithfully flat $R$-scheme of  finite type. Set $X=\tX\otimes_{R} K$.
We remark:
\begin{enumerate}[label=(\arabic*)]
    \item\label{para:setup1} $\tX$ is of finite presentation over $R$, by \cite[I, Corollaire (3.4.7)]{RaynaudGruson},  
       and  each finitely generated ideal sheaf in $\sO_{\tX}$ is also finitely presented, 
      by \cite[I, Th\'eorem (3.4.6)]{RaynaudGruson}. Hence $\sO_{\tX}$ is coherent. 
    \item Let $U\inj X$ be a quasi-compact open immersion. Then any $U$-admissible blow-up $\tY\to \tX$ is again
    faithfully flat and of finite type over $R$.
    \item Recall that an integral domain $A$ is N-2 if its integral closure in any finite field extension of its fraction field 
    is a finite $A$-module.  For example, $R$ is N-2 if $K$ is a complete discrete valuation field,
    or $K$ is algebraically closed.
    \end{enumerate}
\end{para}

\begin{lemma}\label{lem:coh}
In the situation \ref{para:setup} let $U\subset X$ be a dense open normal subscheme and denote by $\tX^{\rm int}$ 
the integral closure of $\tX$ in $U$. Then $\tX^{\rm int}$ is independent of $U$, 
the morphism $\tX^{\rm int}\to \Spec R$ is of finite presentation, and $\sO_{\tX^{\rm int}}$ is coherent.
\end{lemma}
\begin{proof}
We may assume $U$ to be integral and  $\tX=\Spec \tA$. 
As $U$ is normal the scheme $\tX^{\rm int}$ is equal to the normalization of $\tX$ in the function field of $U$, 
which is independent of $U$. As $R$ is a valuation ring which is N-2 it is also  universally Japanese, by \cite[Theorem 5]{Lyu26},
and hence $\tX^{\rm int}$ is finite over $\tX$. Thus $\tX^{\rm int}$ is of finite type and surjective over $R$. 
As the sheaf $\sO_{\tX^{\rm int}}$ has no $R$-torsion the scheme $\tX^{\rm int}$ is flat over $R$, as well.
The statement follows from \ref{para:setup}\ref{para:setup1}.
\end{proof}

\begin{definition}\label{def;Gar}
Let $r\in \Q$ and choose  $m\in \bZ_{>0}$, $n\in \Z$ with $r=n/m$.
	For $(x,v, \epsilon)\in \Spa(X,\tX)$ and $(L,w)$ a finite tame extension of $(k(x),v)$,  we set
	\eq{eq;def;Gar}{\Gar_w= \begin{cases}
    \{a\in \sO_{X,L}^h|\; w(\bar{a}^m/\pi^n)>0\} & \text{if } w(\pi)>0\\
    \{a\in \sO_{X,L}^h|\; w(\bar{a})\ge 0\} & \text{if } w(\pi)=0,
    \end{cases}}
    where $\bar{a}\in L$ denotes the residue class of $a\in \sO_{X,L}^h$, see \ref{para:beta} for the notation.
Note that this does not depend on the choice of $m$ and $n$,
furthermore,  $\Gar_w=\sO^h_{\tX,L,w}$, if $w(\pi)=0$.
If $(L',w')/(L,w)$ is a finite tame extension, then the 
natural map  $\sO_{X,L}^h\to \sO_{X,L'}^h$ maps 
$\Gar_w$ into $\Gar_{w'}$. Hence, applying the construction in \ref{para:beta} to the \'etale sheaf $\sO_X$, we get a tame sheaf
$\Gar\in \Shv_t((X,\tX))$, which in fact is a sheaf of $\sO^t$-modules.
Note that by definition, we have for all $(U,\tU)\in (X,\tX)_t$
\eq{def;Gar2}{\Gar(U,\tU)\subset \Gar(U,\tU[\tfrac{1}{\pi}])= \sO^t(U,\tU[\tfrac{1}{\pi}])=\sO(\tU^{\rm int})[\tfrac{1}{\pi}],}
where the last equality holds by \eqref{eq:Ot}.
\end{definition}

\begin{lemma}\label{lem:Gar3}
Let $(U,\tU)\in (X,\tX)_t$ with $\tU=\Spec \tA$ surjective over $\Spec R$ and $\tA[\tfrac{1}{\pi}]$ normal, 
and $U\subset \tU$ a quasi-compact dense open subset. For $r\in \Q$, we have 
\begin{align*}
\Gar(U,\tU)=\Gar(\tU[\tfrac{1}{\pi}],\tU) & =
\left\{a\in \tA[\tfrac{1}{\pi}]\mid  v_x(\bar{a})>r, \text{ for all }x\in U_{(0)}\right\}\\
 & =\left\{a\in \tA[\tfrac{1}{\pi}] \mid \tfrac{a^m}{\pi^n}\in \sqrt{\pi \tA^{\rm int}}\right\},    
\end{align*}
where $U_{(0)}$ denotes the set of closed points of $U$, 
$v_x$ the unique (real-valued) valuation on the residue field  $K(x)$
extending $v_K$, $\bar{a}$ the image of $a$ in $K(x)$, and  $\tA^{\rm int}$ denotes the integral closure of $\tA$ 
in $\tA[\tfrac{1}{\pi}]$, and $m\in \bZ_{>0}$, $n\in \Z$ with $r=n/m$.
\end{lemma}

Before the proof we recall the following well-known statement.
\begin{lemma}\label{lem:rig-point}
Let $\tA$ be a faithfully flat $R$-algebra of finite type. 
Then there exists a finite field extension $K'/K$ and an $R$-algebra map $\tA\to R'$, which induces
a surjection $\tA\otimes_R K\to K'$, where $R'$ is the valuation ring associated to the unique extension of $v_K$ to $K'$.
\end{lemma}
\begin{proof}
As $\tA$ is a faithfully flat $R$-algebra it has no $\pi$-torsion and $\tA/\pi \tA$ is not the zero ring.
Hence $\hat{\tA}=\varprojlim_{n}(\tA/\pi^n\tA)$ is not the zero ring as it surjects onto $\tA/\pi \tA$. 
Moreover, $\hat{\tA}$ has no $\pi$-torsion.\footnote{If $a=(a_n \text{ mod } \pi^nA)_n\in \hat{\tA}$ is $\pi$-torsion, 
then $a_{n}=\pi^{n-1}b_n$, hence $a=0$.} Thus $\hat{\tA}\otimes_R K$ is a non-zero affinoid $K$-algebra and therefore 
has a maximal ideal. By \cite[8.3, Lemma 6 and Lemma 2]{Bosch-RigGeom} we get a finite field extension $K'/K$ 
and an $R$-algebra map $\hat{\tA}\to R'$ which induces a surjection after $\otimes_R K$. 
We get the map from the statement by precomposing with $\tA\to\hat{\tA}$. 
As generators of the $R$-algebra $\tA$ also topologically generate $\hat{\tA}$, 
their residue classes generate the finite field extension $K'/K$ and thus $\tA\otimes K\to K'$ is surjective. 
\end{proof}

\begin{proof}[Proof of  Lemma \ref{lem:Gar3}.]
As $\tU[\tfrac{1}{\pi}]$ is normal, so is $U$ and the integral closure $\tU^{\rm int}$ of $\tU$ in $U$
coincides with the normalization of $\tU$ and with the integral closure of $\tU$ in $\tU[\tfrac{1}{\pi}]$.
By Lemma \ref{lem:coh}, the $R$-scheme $\tU^{\rm int}$ is again  faithfully flat and of finite presentation.
As $\Gar(U,\tU)=\Gar(U,\tU^{\rm int})$ and $\Gar(\tU[\tfrac{1}{\pi}],\tU)=\Gar(\tU[\tfrac{1}{\pi}],\tU^{\rm int})$, 
see \ref{def;tametopology-intro}\ref{def;tametopology-intro4},
we can therefore assume $\tA=\tA^{\rm int}$ and  furthermore that $\tA$ is a domain.
From the definition, we get inclusions
\begin{multline*}
 \left\{a\in \tA[\tfrac{1}{\pi}] \mid \tfrac{a^m}{\pi^n}\in \sqrt{\pi \tA}\right\} \subset
 \Gar(\tU[\tfrac{1}{\pi}],\tU)\subset \Gar(U,\tU)\\
 \subset  \left\{a\in \tA[\tfrac{1}{\pi}]\mid  v_x(\bar{a})>r, \text{ for all }x\in U_{(0)}\right\},
\end{multline*}
where the second inclusion is induced by restriction, 
and for the third inclusion we use \eqref{def;Gar2}, $v_x(\pi)=1$, and  
that a map $\Spec K(x)\to U$ extends to $\Spec \sO_{v_x}\to \tU$, as $\tU$ is faithfully flat  over $\Spec R$, 
see \cite[8.3, Lemma 6 and Lemma 2]{Bosch-RigGeom} (cf. the proof of Lemma \ref{lem:rig-point}).
Thus it remains to prove:
\begin{claim}\label{lem:Gar3-claim}
Assume $a\in \tA[\tfrac{1}{\pi}]$ satisfies  
\eq{lem:Gar3-claim1}{a^m/\pi^n\not\in \sqrt{\pi\tA}.} 
Then there exists a closed point $x\in U$ such that $v_x(a)\le r$.
\end{claim}
We prove the claim. 
Set $A:=\tA[1/\pi]$ and $f:=a^m/\pi^n \in A$. 
We want to spread out $A[1/f]$ to a faithfully flat $R$-algebra of finite type.
To this end, set $\tC:=\varprojlim \tA/\pi^n\tA$, and $C:=\tC[1/\pi]$ and 
denote by $|-|_{\rm sup}$ the supremum norm on the affinoid $K$-algebra $C$,
and by  $\hat{f}$ the image of $f$ under the natural map $A\to C$.
By \eqref{lem:Gar3-claim1} $\hat{f}$ is not nilpotent in $C$ 
and hence $|\hat{f}|_{\rm sup}\neq 0$, by \cite[3.1, Proposition 10]{Bosch-RigGeom}.
Thus, by \cite[3.1, Proposition 16]{Bosch-RigGeom} there exist $e\ge 1$ and $\lambda\in K^\times$ with
$|\lambda|=|\hat{f}^e|_{\rm sup}$. Set $f_1:=\lambda^{-1} f^e\in A$. Thus $|\hat{f}_1|_{\rm sup}=1$.  
Denote the $\tA$-subalgebra of $A$ generated by $1/f_1$  by $\tB:=\tA[1/f_1]$ and set $\tW:=\Spec \tB$. 
Recall from the beginning of the proof that  $\tU$ is flat and of finite type over $R$ (as $\tA\subset \tA[1/\pi]$ it has no $R$-torsion and hence is flat) and thus the open subscheme $\tW$ is flat and of finite type over $R$ as well. 
Also note  $\tB[1/\pi]=A[1/f]$.
Moreover, $\tW$ also surjects onto $\Spec R$. Indeed, else $\pi$ is a unit in $\tB$, 
i.e., $\pi\cdot \tfrac{\alpha}{f_1^\nu}=1$ in $\tB$, for some $\alpha\in\tA$ and $\nu \ge 0$, 
hence $f_1\in \sqrt{\pi\tA}$.
But then $(f_1^n)_n$ would $\pi$-adically converge to zero, contradicting 
$|\hat{f}_1^n|_{\rm sup}=|\hat{f}_1|^n_{\rm sup}=1$. Thus $\tW$ is integral and faithfully flat  of finite type over $R$. 

Furthermore, $W=\Spec B$, with $B=\tB[1/\pi]=A[1/f]$, is open dense in $U$.
Thus we find a non-zero element $h\in B$, such that $\Spec B[1/h]\subset W\cap U$. 
Now replacing in the above discussion $A$ by $B$ and $f$ by $h$, shows that up to replacing $h$ by $\mu^{-1}h^j$, 
for some $\mu\in K^\times$, we can assume that $\tB[1/h]$, 
the $\tB$-subalgebra of $B$ generated by $1/h$, is faithfully flat of finite type over $R$.

By Lemma \ref{lem:rig-point} we find a closed point $x\in \Spec B[1/h]\subset W\cap U$ with residue field 
$K(x)$ finite over $K$, which induces a morphism
$\tB[1/h]\to R'$ with $R'$ the valuation ring corresponding to $v_x$ the extension of $v_K$ to $K(x)$.
As in particular $1/f_1\in \tB$, we find $v_x(\bar{f}_1)\le 0$. 
We have $|\lambda|=|\hat{f}^e|_{\rm sup}\ge 1$, where the equality holds by the choice of $e$ and $\lambda$, and  
the inequality holds by \eqref{lem:Gar3-claim1}. 
Thus, we find  
\[ev_x(\bar{f})\le v_x(\lambda)=v_K(\lambda)\le 0,\] 
where the first inequality is implied by $v_x(\bar{f_1})\le 0$ and the last inequality holds by applying $\log_c$ to $|\lambda|\ge 1$.
As $e \ge 1$ and $f=a^m/\pi^n$, Claim \ref{lem:Gar3-claim} follows from the fact that $v_x(\pi)=v_K(\pi)=1$.
This completes the proof of Lemma \ref{lem:Gar3}.
\end{proof}
 
\begin{lemma}\label{lem:Gar-shift}
Let the situation be as in \ref{para:setup} above.
Assume $X$ is normal. 
Let $U\subset X$ be a quasi-compact dense open subset.
\begin{enumerate}[label=(\arabic*)]
    \item\label{lem:Gar-shift1} $\Gar_{(U,\tX)}=\Gar_{(X,\tX)}$, for $r\in \Q$.
    \item\label{lem:Gar-shift2}
    Multiplication by $\pi^n$ ($n\in \Z$) induces an isomorphism of $\sO_{\tX_{\et}}$-modules
    \[\G_a(0)_{(U,\tX)}\xrightarrow{\simeq} \G_a(n)_{(U,\tX)}.\]
    \item\label{lem:Gar-shift3} In case the ordered group $v_K(K^\times)$ is divisible, 
    multiplication by $\pi^r$ ($r\in \Q$) induces an isomorphism of $\sO_{\tX_{\et}}$-modules
    \[\G_a(0)_{(U,\tX)}\xrightarrow{\simeq} \Gar_{(U,\tX)}.\] 
    \end{enumerate}
    \end{lemma}
\begin{proof}
    This follows directly from Lemma \ref{lem:Gar3}.
\end{proof}

\begin{corollary}\label{cor:Gar} 
Let the situation be as in \ref{para:setup} above.
Assume $X$ is  normal. 
\begin{enumerate}[label=(\arabic*)]
    \item\label{cor:Gar0} The sheaf $\G_a(0)_{(X,\tX)}$ is a quasi-coherent $\sO_{\tX_{\et}}$-module.
    \item\label{cor:Gar1}  Let $\fXrig$ be the rigid space associated to the formal completion of $\tX$ along $V(\pi\sO_{\tX})$. 
            Then there is a canonical isomorphism of $\cO^{\circ}_{\fXrig}$-modules
            \[\G_a(0)_{\fXrig}\cong \sO(1),\]
            where  the left hand side is defined as in \eqref{eq:rig2}
            and  the right hand side is the $\sO^\circ_{\fXrig}$-module from Definition \ref{defn:Or}.    
    \item\label{cor:Gar2} Assume $R$ is a complete discrete valuation ring.
    Then the sheaf $\G_a(0)_{(X, \tX)}$ is a coherent $\sO_{\tX_{\et}}$-module.
    \item\label{cor:Gar3} Assume $K$ is algebraically closed. Then, $\G_a(0)_{(X,\tX)_{\Zar}}$ 
    is an almost coherent $\sO_{\tX}$-module in the sense of \cite[Definition 4.1.10]{Zavyalov}.
    \end{enumerate}
\end{corollary}
\begin{proof}
\ref{cor:Gar0} Let $\tW=\Spec\tB\to \tV=\Spec \tA$ be a map of \'etale  $\tX$-schemes and set $W=\tW\times_{\tX} X$ and $V=\tV\times_{\tX} X$. We have to show that the natural map
\eq{cor:Gar3.1}{\G_a(0)(V, \tV)\otimes_{\tA} \tB\to \G_a(0)(W, \tW)}
is an isomorphism. 
By Lemma \ref{lem:Gar3}, 
\[\G_a(0)(V, \tV)=\sqrt{\pi \tA^{\rm int}}\quad \text{and}\quad \G_a(0)(W, \tW)=\sqrt{\pi \tB^{\rm int}},\]
where $\tA^{\rm int}$ is the integral closure of $\tA$ in $\tA[\tfrac{1}{\pi}]$
and $\tB^{\rm int}$ is the integral closure of $\tB$ in $\tB[\tfrac{1}{\pi}]=\tB\otimes_{\tA} \tA[\tfrac{1}{\pi}]$.
By \cite[\href{https://stacks.math.columbia.edu/tag/03GE}{Lemma 03GE}]{stacks-project},
\eq{cor:Gar3.5}{\tB^{\rm int}=\tA^{\rm int}\otimes_{\tA} \tB.} 
In particular the vertical maps in the following diagram are \'etale
\[\xymatrix{
\tB^{\rm int}\ar[r] & \frac{\tB^{\rm int}}{\pi \tB^{\rm int}}\ar[r] & \frac{\tB^{\rm int}}{\sqrt{\pi A^{\rm int}}\cdot \tB^{\rm int}}\\
\tA^{\rm int}\ar[r]\ar[u] & \frac{\tA^{\rm int}}{\pi \tA^{\rm int}}\ar[r]\ar[u] & \frac{\tA^{\rm int}}{\sqrt{\pi A^{\rm int}}}.\ar[u]
}\]
Thus the ring in the upper right corner is reduced and hence 
\[\sqrt{\pi \tB^{\rm int}}=\sqrt{\pi A^{\rm int}}\cdot \tB^{\rm int}=
\sqrt{\pi A^{\rm int}}\otimes_{\tA^{\rm int}} \tB^{\rm int} = \sqrt{\pi A^{\rm int}}\otimes_{\tA} \tB, \]
where the second equality holds by flatness of  the map $\tA^{\rm int}\to \tB^{\rm int}$ 
and the third equality by \eqref{cor:Gar3.5}.
Hence \eqref{cor:Gar3.1} is an isomorphism.

\ref{cor:Gar1} It suffices to show that both sides are equal when evaluated on 
affinoid subdomains $\fU^{\rig}=\Sp(\hat{\tA}[\tfrac{1}{\pi}])$,
where $\fU=\Spf \hat{\tA}$ is the completion along $V(\pi \tA)$ of 
$\tU=\Spec \tA\subset \tX'$, for some $\tX'\in \Sigma_{\tX}$. 
Set $B:=\sO_{\fXrig}^o(\fU^{\rig})$ and denote by $\tA^{\rm int}$ the integral closure of $\tA$ in  $A=\tA[\tfrac{1}{\pi}]$.
Let $U=\Spec A$. By Lemma \ref{lem:sectionsFrig} and Lemma \ref{lem:Gar3}, the sheaf $\G_a(0)_{\fXrig}$ is  associated to
the presheaf which on affinoid subdomains $\fU^{\rig}$ as above is given by
\[\G_a(0)(U,\tU)\otimes_{\tA^{\rm int}} B= 
\left\{a\in A\mid  v_x(\bar{a})>0, \text{ for all }x\in U_{(0)}\right\} \otimes_{\tA^{\rm int}} B,\]
where $\bar{a}$ denotes the various images of $a$ in the residue fields $K(x)$,  for $x\in U_{(0)}$.
The natural surjection $A\to K(x)$ is continuous in the $\pi$-adic topology and as $K(x)$ is  complete,
it factors via $A\to \hat{\tA}[\tfrac{1}{\pi}]$, hence $x$ defines a point $\hat{x}\in \fU^{\rig}$. 
On the other hand a point $y\in \fU^{\rig}$ defines a surjection 
$\hat{\tA}[\tfrac{1}{\pi}]\to K(y)$  with $K(y)$ finite over $K$.  
As generators of the $R$-algebra $\tA$ topologically generate $\hat{\tA}$, they also 
generate the finite field extension $K(y)/K$. Hence the composition 
$A\to \hat{\tA}[\tfrac{1}{\pi}]\to K(y)$ is surjective as well.
This defines a closed point $y_0\in U$.
It is clear that $\widehat{(y_0)}=y$ and $(\hat{x})_0=x$. Thus we get a bijection 
$U_{(0)}\overset{1:1}{\longleftrightarrow}\fU^{\rig}$ 
and hence
\[\G_a(0)(U,\tU)\otimes_{\tA^{\rm int}} B= 
\left\{a\in A \mid  v_x(\bar{a})>0, \text{ for all }x\in \fU^{\rig}\right\} \otimes_{\tA^{\rm int}} B.\]
The natural map $\tA^{\rm int}\to B$ is flat, by Lemma \ref{lem:Orig-flat-Ot}, therefore the inclusion \[
\left\{a\in A \mid  v_x(\bar{a})>0, \text{ for all }x\in \fU^{\rig}\right\}\subseteq \tA^{\rm \int}
\]
stemming from the last equality in Lemma \ref{lem:Gar3}, 
gives an injective map  induced by the multiplication
\begin{equation}\label{eq:inj-partial}
\left\{a\in A \mid  v_x(\bar{a})>0, \text{ for all }x\in \fU^{\rig}\right\} \otimes_{\tA^{\rm int}} B\inj B,
\end{equation}
Note that for $a\in A$ with $v_x(\bar{a})>0$, for all $x\in \fU^{\rig}$, and 
$b\in B$ the product $\hat{a}b\in B$, with $\hat{a}$ the image of $a$ under $A\to \hat{\tA}[\tfrac{1}{\pi}]$, 
satisfies\[    
|\hat{a}b|_{\rm sup}\le |\hat{a}|_{\rm sup} |b|_{\rm \sup}\le |\hat{a}|_{\rm sup}= c^{v_x(\bar{a})}< 1, \quad 
\text{for some } x\in \fU^{rig}.
\]
Here we use the Maximum Principle to guarantee the existence of an $x$ 
such that the last equality is satisfied. 

Thus the map \eqref{eq:inj-partial} factors through
\eq{cor:Gar4}{
\G_a(0)(U,\tU)\otimes_{\tA^{\rm int}} B\inj 
\cO(1)(\fU^{\rig})=\{b\in \hat{\tA}[\tfrac{1}{\pi}]\mid |b|_{\rm sup}<1\}\subset B.}
It remains to show that this map is surjective. 
To this end, let $b\in \cO(1)(\fU^{\rig})$. We find $a\in A$ and $c\in \hat{\tA}$, such that 
$b=\hat{a}+\pi c_1$, where $\hat{a}\in \hat{\tA}[\tfrac{1}{\pi}]$ is the image of $a$ under the natural map 
$A\to \hat{\tA}[\tfrac{1}{\pi}]$ and $c_1$ is the image of $c$ under $\hat{\tA}\to B$. 
As $|b|_{\rm sup}<1$ and $|\pi c|_{\rm sup}<1$ we get $|\hat{a}|_{\rm sup}<1$. 
Thus for all $x\in U_{(0)}=\fU^{\rig}$ we have $v_x(\bar{a})=v_x(\ol{\hat{a}})>0$. 
Thus $a\in \G_a(0)(U,\tU)$, by Lemma \ref{lem:Gar3}. Hence $b$ is the image of 
\[a\otimes 1 + \pi\otimes c_1\in \G_a(U,\tU)\otimes_{\tA^{\rm int}} B \]
and \eqref{cor:Gar4} is surjective.
In view of Lemma \ref{lem:sectionsFrig}  the isomorphism \eqref{cor:Gar4}  induces  
the isomorphism $\G_a(0)_{\fXrig}\xrightarrow{\simeq} \cO(1)$ from the statement.

\ref{cor:Gar2} Let $(V, \tV)=(\Spec \tA[\tfrac{1}{\pi}], \Spec \tA)$ with $\tV\to \tX$ \'etale.
As $\tA$ is noetherian and $\tA^{\rm int}$ is a finite $\tA$-module, 
see Lemma \ref{lem:coh}, the $\tA$-module $\tA^{\rm int}$ is in fact noetherian. 
Hence $\G_a(0)(V,\tV)=\sqrt{\pi\tA^{\rm int}}$ is a finitely generated $\tA$-module. 
Together with \ref{cor:Gar0}, we get coherence.

\ref{cor:Gar3} Assume $K$ is algebraically closed.
 Let $(V,\tV)=(\Spec A, \Spec \tA)$ be as above with $A=\tA[\tfrac{1}{\pi}]$ 
 and let $\tA^{\rm int}$ be the integral closure of $\tA$ in $A$.
Set
 \[M_V:=\G_a(0)(V,\tV)\quad \text{and}\quad M_V(N):=\{a\in A \mid a^{N}\in \pi\tA^{\rm int}\}, \quad \text{for }N\ge 1.\]
By Lemma \ref{lem:Gar3} and definition, we have 
\[ M_V(N)\subset M_V,\quad  M_V(N)\subset M_V(N'), \text{ for } N\le N',\quad \text{and}\quad   \cup_{N\ge 1} M_V(N)=M_V.\]
Denote by  $\pi^{1/N}$ a fixed root of the polynomial  $T^{N}-\pi\in K[T]$ in $K$. Then the inclusion
\eq{cor:Gar7}{\pi^{1/N}\tA^{\rm int}\inj M_V(N)}
is an equality. Indeed, for any $a\in M_V(N)$, we have $(a/\pi^{1/N})^{N}\in \tA^{\rm int}$, hence $a/\pi^{1/N}\in \tA^{\rm int}$.
Thus $M_V(N)$ is a free $\tA^{\rm int}$-module of rank one. As $\tA^{\rm int}$ is finite over $\tA$ 
(see the proof of Lemma \ref{lem:coh}), it follows that $M_V(N)$ is a $\pi$-torsion free finite $\tA$-module.
Thus $M_V(N)$ is an almost finitely presented  $\tA$-module, by \cite[Lemma 2.12.4]{Zavyalov}  
(with $(R,I\subset\fm)$ there, taken as $(\tA, \pi\tA\subset \fm \tA)$ here). As the ring $\tA$ is coherent, 
see \ref{para:setup}\ref{para:setup1}, 
we get that $M_V(N)$ is almost coherent, by \cite[Lemma 2.6.13 and Corollary 2.6.15]{Zavyalov}. 

We show that $M_V$ is an almost coherent $\tA$-module.
Indeed, let $\fm_0\subset \fm\tA$ be a finitely generated ideal. Then, there exists a $\delta_0\in \fm\subset R$ 
and a natural number $N_0\ge 1$
such that $\fm_0\subset \delta_0 \tA$ and $\delta_0^{N_0}\in \pi R$, i.e., $\delta_0\in \pi^{\frac{1}{N_0}} R$.
Then, we have 
\[\fm_0\cdot M_V\subset \delta_0\cdot M_V\subset \pi^{\frac{1}{N_0}} \cdot M_V\subset 
 \pi^{\frac{1}{N_0}} \cdot \tA^{\rm int}= M_V(N_0).\]
Thus the cokernel of the inclusion $M_V(N_0)\inj M_V$ is annihilated by $\fm_0$.
As $M_V(N_0)$ is almost coherent as shown above, \cite[Lemma 2.6.5(2)]{Zavyalov} yields that 
$M_V$ is an almost coherent  $\tA$-module as well. Thus by \cite[Remark 4.1.2 and Lemma 4.1.11]{Zavyalov}
the $\sO_{\tX}$-module  $\Gar_{(X, \tX)_{\Zar}}$ is almost coherent.
\end{proof}

The following is an integral variant of GAGA for a  particular sheaf (cf. \cite[Proposition \FKC{II}.9.4.2]{FK}):

\begin{thm}\label{thm;comparison-final}
Assume $K$ is either a complete discrete valuation field or a complete non-archimedean algebraically closed field with
ring of integers $R=\sO_K$. Let $\tX$ be a proper  and faithfully flat $R$-scheme 
such that $X=\tX\otimes_{\OK} K$ is normal. 
Denote by $\fX^{\rm rig}$ the rigid space associated to the formal completion $\fX$ of $\tX$ along $V(\pi)$. 
\begin{enumerate}[label=(\arabic*)]
    \item\label{thm;comparison-final1} There exists a canonical equivalence
\[R\Gamma((X,\tX)_t,\Gar) \simeq R\Gamma(\fXrig,\cO(c^r)),\quad \text{for }
\begin{cases}
    r\in \Z & \text{if } K \text{ is a cdvf},\\
    r\in \Q & \text{if }K=\overline{K}.
\end{cases}\]
\item\label{thm;comparison-final2} Assume $\ch(K)=0$ or $\dim X\le 3$.
Let $U\subseteq X$ be a dense open such that $X$ is regular in a neighborhood of $X-U$.
Then, there exists a canonical equivalence
\[R\Gamma((U,\tX)_t,\Gar) \simeq R\Gamma(\fXrig,\cO(c^r)),\quad \text{for }
\begin{cases}
    r\in \Z & \text{if } K \text{ is a cdvf},\\
    r\in \Q & \text{if }K=\overline{K}.
\end{cases}\]
\end{enumerate}
In particular, under the assumptions in \ref{thm;comparison-final2}, we have an equivalence
\eq{eq;thm;comparison-final}{
 R\Gamma_t((U,\tX),\Gar)\simeq R\Gamma_t((V,\tX),\Gar),}
for $r$ as above and for any two dense open subsets  $V\subset U\subset X$
such that $U_{\mathrm{sing}}\subset V$.
\end{thm}
\begin{proof}
By Lemma \ref{lem:Gar-shift}\ref{lem:Gar-shift2} and  \ref{lem:Gar-shift3} it suffices to consider $r=0$.
If $K$ is a cdvf, then $\G_a(0)\in \ShvCoh{(X,\tX)}$, by  Corollary \ref{cor:Gar}\ref{cor:Gar2}, 
and hence statement \ref{thm;comparison-final1}  holds by Theorem \ref{thm;comparison} and Corollary \ref{cor:Gar}\ref{cor:Gar1};
if $K=\overline{K}$, then $\G_a(0)\in \ShvaCoh{(X,\tX)}$, by  Corollary \ref{cor:Gar}\ref{cor:Gar3}, 
and hence statement \ref{thm;comparison-final1} holds by Theorem \ref{thm:almost-comparison}.

By Lemma \ref{lem:Gar-shift}\ref{lem:Gar-shift1} the sheaf $\G_a(0)$ is tamely birational (see Definition \ref{defn:tb}). 
Hence, in the situation of Theorem \ref{thm:colim-smooth-hironaka} \ref{thm;comparison-final2}
(resp. Theorem  \ref{thm:almost-colim-smooth-hironaka}) yields an equivalence
\[
R\Gamma((U,\tX)_t,\G_a(0)) \simeq \colim_{\tY\to \tX} R\Gamma(\fY^{\rm rig},\cO(1)),
\]
where $\tY\to \tX$ ranges over $U$-admissible blow-ups in a center $\tT$ such that $T:=\tT\otimes_{R} K$ is regular.
Therefore, the first statement of \ref{thm;comparison-final2} follows from Lemma \ref{lem2;CohRigid} below, and \eqref{eq;thm;comparison-final} follows from the fact that
$U_{\mathrm{sing}}\subset V$, therefore $X$ is regular in a neighborhood of $X-V$.
\end{proof}

\def\NSchKp{\mathrm{NSch}_K'}

\begin{rmk}\label{rmk;thm;comparison-final}
Let $\NSchKp$ be the category of normal schemes $X$ separated of finite type over $K$ such that  $X_{\mathrm{sing}}$ is proper over $K$.
Consider the functor (cf. \eqref{eq;cFr}):
\[ \NSchKp\ni  X\mapsto \cF(X) := R\Gamma_t(X/R,\G_a(0)).\]
Let $X\in \NSchKp$ and assume $\ch(K)=0$ or $\dim X\le 3$.
\begin{itemize}
\item[(1)]
There is a normal compactification $X\hookrightarrow \Xb$ over $K$ such that  $\Xb$ is regular in a neighborhood of $\Xb-X$: Take any
normal compactification $X\hookrightarrow \Xb$ over $K$.
Since $X_{\mathrm{sing}}$ is proper over $K$, its closure in $\Xb$ is contained in $X$. Hence, we can resolve the singularities of $\Xb$ along $\Xb-X$ by $X$-admissible blowups to make $\Xb$ regular in a neighborhood of $\Xb-X$. 
By Theorem \ref{thm;comparison-final}\ref{thm;comparison-final2}, this implies
\[\cF(X)\otimes_R K \simeq R\Gamma(\Xb,\cO).\] 
\item[(2)]
For a morphism $f:Y\to X$ in $\NSchKp$ which is an isomorphism over a dense open $U\subset X$ such that $X_{\mathrm{sing}}\subset U$ and $Y_{\mathrm{sing}}\subset f^{-1}(U)$, we have $\cF(X)\simeq \cF(Y)$.
Indeed, this follows from a commutative diagram
\[\xymatrix{
\cF(X)\ar[r] \ar[d]^{f^*} & \cF(U)\ar[d]\\
\cF(Y)\ar[r] & \cF(U\times_X Y)\\ }\]
where the right vertical map is an isomorphism by the assumption and the horizontal maps are isomorphisms thanks to \eqref{eq;thm;comparison-final}.
Noting $\Sm_K\subset \NSchKp$, this implies the property \ref{thm;main2} \ref{thm;main2c} of Theorem \ref{thm;main}.

\end{itemize}
\end{rmk}

\def\fX{\mathfrak{X}}
\def\fY{\mathfrak{Y}}
\def\fU{\mathfrak{U}}
\def\frig{f^{\rig}}
\def\an{\mathrm{an}}

\begin{lemma}\label{lem2;CohRigid}
Let $\tX$ be a proper and faithfully flat $R$-scheme and assume $X=\tX\otimes_R K$ is regular.
Let $\tY\to \tX$ be the blow-up in a center $\tT$  such that $T:=\tT\otimes_{R} K$ is regular. 
Let $f: Y=\tY\otimes_R K \to X$ be the induced map  of regular $K$-schemes and $\frig: \fYrig\to \fXrig$ be the induced map on the rigid spaces associated to the formal completions of $\tX$ and $\tY$ along $V(\pi)$.
Then, pullback induces an equivalence  
\[\cO_{\fXrig}(s)\simeq R\frig_*\cO_{\fYrig}(s),\quad \text{for all } s\in \R_{>0}.\] 
\end{lemma}
\begin{proof}
 As $f: Y\to X$ is a blow-up of a regular $K$-scheme in a regular center, we have $\cO_X=Rf_*\cO_Y$.
 Furthermore, the fibers of $f$ are projective spaces. Hence the statement follows from Lemma \ref{lem:hdi-general} below, recalling that if $X$ is proper then $X^{\rm an}=\fX^{\rig}$ by \cite[Proposition \FKC{II}.9.1.5]{FK}.
\end{proof}
\def\fan{f^{\rm an}}
\def\Xan{X^{\rm an}}
\def\Yan{Y^{\rm an}}
\begin{lemma}\label{lem:hdi-general}
Let $f:Y\to X$ be a proper morphism of $K$-schemes separated of finite type and let 
$\fan: \Yan\to \Xan$ the induced map on the rigid spaces associated to the analytification of $X$ and $Y$ as in \cite[Construction \FKC{II}.9.1.6]{FK}.
Assume $f$ satisfies
\begin{enumerate}[label=(\arabic*)]
    \item\label{lem:hdi-general1} $Rf_*\sO_Y\simeq \sO_X$, and
    \item\label{lem:hdi-general2} for all $x\in X$ there is an $n\ge 0$ (depending on $x$) such that $f^{-1}(x)\cong \P^n_x$.
\end{enumerate}
Then, pullback induces  an equivalence  
\[\cO_{\Xan}(s)\simeq R\fan_*\cO_{\Yan}(s),\quad \text{for all } s\in \R_{>0}.\] 
\end{lemma}
\begin{proof}
Consider the morphism of exact triangles 
\eq{eq0;lem2;CohRigid}{ 
\xymatrix{ 
R\fan_*\cO_{\Yan}(s)\ar[r] &  R\fan_*\cO_{\Yan}\ar[r] & R\fan_*\cO_{\Yan}(s,\infty)\ar[r] &\\
\cO_{\Xan}(s)\ar[r]\ar[u]^{\alpha} &  \cO_{\Xan}\ar[r]\ar[u]^{\beta} & \cO_{\Xan}(s,\infty)\ar[r]\ar[u]^{\gamma} &,
}}
where $\cO_{\Xan}(s,\infty)=\cO_{\Xan}/\cO_{\Xan}(s)$.
As the natural map $\cO_X\to Rf_*\cO_Y$ is an equivalence by assumption \ref{lem:hdi-general1}
so is $\beta$ by GAGA (see \cite[Theorem \FKC{II}.9.4.1 and Proposition \FKC{II}.9.1.5]{FK}).
Thus it remains to show that $\gamma$ is an equivalence.

As $\cO_{\Yan}(s,\infty)$ is overconvergent, by \cite[Lemma 1.5.2]{vdP},  so is  
$R^i\fan_*\cO_{\Yan}(s,\infty)$, for all $i\ge 0$, by \cite[Proposition 2.4.1]{dJvdP}.
Thus $\gamma$  is an equivalence if its stalk $\gamma_a$ is an equivalence at all analytic points $a$ of $\Xan$, 
by  \cite[\S2.1 and Lemma 2.3.2, 6.]{dJvdP}.
By the base change theorem \cite[Theorem 2.7.4]{dJvdP} and \cite[Lemma 3.16]{vdP} the stalk $\gamma_a$ is the map
\eq{eq3;lem2;CohRigid}{\frac{F_a}{\{f\in F_a\mid |f|_a<s\}}=\cO(s,\infty)(F_a)\to R\Gamma(\Yan_a,\cO_{\Yan_a}(s,\infty)),}
where $F_a$ is the completed residue field of the analytic point $a$ and $\Yan_a$ denotes the fiber of $\fan$ over $a$.
This latter map fits in a diagram like \eqref{eq0;lem2;CohRigid} with $\fan: \Yan\to \Xan$ replaced by
$\Yan_a\to \Sp\, F_a$. Noting $\Yan_a$ is a rigid projective space $\P^{n,\rig}_{F_a}$ over $F_a$, by 
assumption \ref{lem:hdi-general2} and \cite[Proposition \FKC{II}.9.1.10]{FK}, 
it follows from Lemma \ref{lem1;CohRigid} and GAGA that \eqref{eq3;lem2;CohRigid} is an equivalence. 
This completes the proof of Lemma  \ref{lem:hdi-general} and hence also of Theorem \ref{thm;comparison-final}.
\end{proof}

\begin{thm}\label{thm;A^1-invariance}
Let $K$ and $U\subset X\subset \tX$ be as in Theorem \ref{thm;comparison-final}\ref{thm;comparison-final2}.
Assume further that $X$ is regular.
For any dense open $V\subset \P^1_U$, the pullback map
\[R\Gamma((U,\tX)_t,\Gar) \to R\Gamma((V,\P^1_{\tX})_t,\Gar)\]
is an equivalence, where $r\in \Q$ has the same constraints as in Theorem \ref{thm;comparison-final}\ref{thm;comparison-final2}. 
\end{thm}
\begin{proof}
By Theorem \ref{thm;comparison-final}, we have a commutative diagram
\[\xymatrix{
	R\Gamma((U,\tX)_t,\Gar) \ar[r]\ar[d]^\simeq  &R\Gamma((V,\P^1_{\tX})_t,\Gar)\ar[d]^\simeq\\
	R\Gamma(\fXrig,\cO(c^r)) \ar[r] & R\Gamma(\fXrig\times \P_K^{1,\rig},\cO(c^r)).}\]
Thus  Lemma \ref{lem:hdi-general} applied to the projection $f: X_K\times \P^1_K\to X_K$ together with \cite[Proposition \FKC{II}.9.1.5]{FK}
yields the statement.
\end{proof}

\begin{defn}\label{def;ModR}
	Let $K$ be as in \ref{para:setup}, $R=\cO_K$  and $\pi$ a pseudo-uniformizer.
	We denote by $\Mod_R^0$ the full subcategory of the category of $R$-modules $\Mod_R$
    consisting of such $M$ that are annihilated by a power of $\pi$.
    This is a Serre category and we define $\cD(R)^0$ to 
    be the full subcategory of $\cD(R)$ of complexes whose cohomology groups are in $\Mod_R^0$.
    
    We denote by $\Mod_R^f$ the full subcategory of $\Mod_R$ consisting of such $M$ that have a finitely generated 
    $R$-submodule $M'\subseteq M$ with $M/M'\in \Mod_R^0$. 
\end{defn}

\begin{rmk}\label{rmk:finite-ubt}
\begin{enumerate}[label=(\arabic*)]
    \item\label{rmk:finite-ubt1} Let $K$ be a cdvf and $\pi\in R$ a uniformizer. Then
    $M\in \Mod_R^f$ if and only if the $R$-torsion submodule $M_{\rm tor}$ is annihilated by a fixed power of $\pi$ 
    and $M/M_{\rm tor}$ is finitely generated. 
    
    Indeed, if $M/M_{\rm tor}$ is finitely generated, it is free so the quotient map splits and we can choose $M'$ to be $M/M_{\rm tor}$.
    For the other implication, assume $M'$ is a finitely generated $R$-module such that $M/M'$ is annihilated by $\pi^e$, 
    for some $e\ge 1$. Then $\pi^e\cdot (M/M_{\rm tor})\subset (M'/M'\cap M_{\rm tor})=:N'$. As $R$ is noetherian, so is the $R$-module  
    $N'$ and thus $M/M_{\rm tor}$ is finitely generated as it is isomorphic to a submodule of $N'$. Moreover,
    we get that $\pi^e M_{\rm tor}$ is contained in the finitely generated $R$-module $M_{\rm tor}\cap M'$ and hence there exists an
    $n\ge e$ such that $\pi^n M_{\rm tor}=0$.
    \item\label{rmk:finite-ubt2} 
    If $K=\ol{K}$, then any almost finitely generated $R$-module (e.g. \cite[Definition 2.5.1]{Zavyalov}) lies in $\Mod_R^f$. 
    The converse however is not true.
\end{enumerate}
\end{rmk}

\begin{rmk}\label{rmk:lattice-modf}
    Let $M\in \Mod_R^f$ and let $f\colon M\to M\otimes_R K$ be the canonical map. Then $f(M)$ is an $R$-lattice in the sense of \cite[VII, No.1, Definition 1]{BourbakiCA}. Indeed, let $M'\subseteq M$ be finitely generated such that $\pi^n M/M' = 0$. Then $E\subset f(M)\subset \pi^{-n}E$  with $E=f(M')$, which is a finitely generated free $R$-module 
    as $R$ is a valuation ring.
\end{rmk}

\begin{lemma}\label{lem:modf}
The category $\Mod_R^f$ is closed under quotients. Let $N\xrightarrow{g} M\xrightarrow{f} C\to 0$ be an exact sequence of $R$-modules such that $N,C\in \Mod_R^f$. Then $M\in \Mod_R^f$.
    \begin{proof}
The first assertion is obvious. By replacing $N$ by the image of $g$, we can suppose that $g$ is injective. Let $i_N\colon N'\subseteq N$ and $i_C\colon C'\subseteq C$ be finitely generated with $N/N'$ and $C/C'$ in $\Mod_R^0$. 
Let $\phi\colon R^n\to C'$ be a surjection of $R$-modules and let $\psi\colon R^n\to M$ be a lift of $i_C\circ \varphi$.
We obtain a commutative diagram with exact rows:\[
        \begin{tikzcd}
            0\ar[r] &N'\ar[r]\ar[d,hook,"i_N"]&N'\oplus R^n\ar[r]\ar[d,"{(g\circ i_
            N)\oplus \psi}"]&R^n\ar[r]\ar[d,"i_C\circ \phi"] &0\\
            0\ar[r] &N\ar[r,"g"]&M\ar[r,"f"]&C\ar[r] &0.
        \end{tikzcd}
        \]
        Letting $M'\subseteq M$ be the image of $(g\circ i_M)\oplus \psi$, it is finitely generated by construction 
        and the snake lemma yields $M/M'\in \Mod_R^0$.
    \end{proof}
\end{lemma}

\begin{thm}\label{thm:fg-uncond}
    Assume $K$ is either a complete discrete valuation field or a complete non-archimedean algebraically closed field with
ring of integers $R=\sO_K$. Let $\tX$ be a proper  and faithfully flat $R$-scheme 
such that $X=\tX\otimes_{\OK} K$ is smooth. 

\begin{enumerate}[label=(\arabic*)]
    \item \label{thm:fg-uncond1} Let $r\in \Q$ have the same constraints as in 
Theorem \ref{thm;comparison-final}\ref{thm;comparison-final1} depending on $K$. Then 
\[H^i((X,\tX)_t,\Gar)\in \Mod_R^f, \quad \text{for all } i.\]
\item\label{thm:fg-uncond2} Let $U\subset X\subset \tX$ be as in Theorem \ref{thm;comparison-final}\ref{thm;comparison-final2}.
Let $r\in \Q$ have the same constraints as in Theorem \ref{thm;comparison-final}\ref{thm;comparison-final2} depending on $K$.
Then

\[H^i((U,\tX)_t,\G_a(r))\in \Mod_R^f,\quad \text{for all }i.\]
\end{enumerate}
    \begin{proof}
    \ref{thm:fg-uncond2} follows from \ref{thm:fg-uncond1} and Theorem \ref{thm;comparison-final}.
    We prove \ref{thm:fg-uncond1}.
    By Lemma \ref{lem:Gar-shift}, we can suppose $r=0$. Fix a pseudo-uniformizer $\pi$ of $R$ and 
    denote for  an $R$-algebra $A$ by $\hat{A}$ its $\pi$-adic completion.
    Consider the canonical map 
    \[
    \chi\colon R\Gamma(\tX,\G_a(0)_{(X,\tX),\Zar})\to R\Gamma((X,\tX)_t,\G_a(0)).\] 
    By Corollary \ref{cor:Gar}\ref{cor:Gar2},\ref{cor:Gar3}, we have $H^i(\tX,\G_a(0)_{(X,\tX),\Zar})\in \Mod_R^f$, 
    where in case $K$ is algebraically closed we use \cite[Lemma 4.4.4 and Theorem 5.1.3]{Zavyalov} and 
    Remark \ref{rmk:finite-ubt}\ref{rmk:finite-ubt2}.
    By Lemma \ref{lem:modf} it suffices to check that $\mathrm{Cone}(\chi)\in \cD(R)^0$. 
    By Theorem \ref{thm;comparison-final}\ref{thm;comparison-final1} 
    and, in case $K$ is a cdvf, by \eqref{eq:GFGA} and Corollary \ref{cor:Gar}\ref{cor:Gar2}
    (resp. in case $K=\ol{K}$, by \eqref{para:ac1} and Corollary \ref{cor:Gar}\ref{cor:Gar3}),
     the cone 
     of $\chi$ is equivalent to the cone of\[
    \chi': R\Gamma(\fX,\G_a(0)_{(X,\tX),\Zar}^{\rm for})\to R\Gamma(\fXrig,\cO(1)),
    \]
    where $\fX$ and $\fXrig$ are as in Theorem \ref{thm;comparison-final}.
    Take a finite open covering $\tX=\cup_{1\leq j\leq m} \tU_j$ by affine schemes $\tU_j=\Spec(A_j)$ and let $\fU_j = \Spf(\hat{A}_j)$, 
    which induces an admissible covering $\fXrig= \cup_{1\leq j\leq m} \fU_j^{\rig}$. Let $\tU_{[\bullet]}$, $\fU_{[\bullet]}$, 
    and $\fU_{[\bullet]}^{\rm rig}$ be the associated \v Cech nerves on $\tX$, $\fX$ and $\fX^{\rm rig}$, respectively, which are $m$-skeletal. Since $\tX$ is separated, $\tU_{[q]}= \Spec(A_{[q]})$ and $\fU_{[q]} = \Spf(\hat{A}_{[q]})$, so $\fU^{\rig}_{[q]}=\Spf(\hat{A}_{[q]})^{\rm rig}$ is affinoid.
	 We have an equivalence
     \[
    \G_a(0)_{(X,\tX),\Zar}^{\rm for}(\fU_{[q]})[0]\xrightarrow{\simeq}R\Gamma(\fU_{[q]},\G_a(0)_{(X,\tX),\Zar}^{\rm for}).
    \]
    In case $K$ is a cdvf, this uses \cite[Theorem \FKC{I}.7.1.1]{FK} together with the fact that the $\cO_{\tX}$-module
     $\G_a(0)_{(X,\tX),\Zar}$ is coherent, and hence $\G_a(0)_{(X,\tX),\Zar}^{\rm for}$ is 
     $\pi$-adically complete and hence 
    adically  quasi-coherent; in case $K=\ol{K}$ this uses \cite[Lemma 4.8.11]{Zavyalov}, 
    the fact that $\G_a(0)_{(X,\tX),\Zar}$ is quasi-coherent almost coherent, and \cite[Corollary 4.6.4]{Zavyalov}, which 
    says that $\G_a(0)_{(X,\tX),\Zar}^{\rm for}=c^*\G_a(0)_{(X,\tX),\Zar}$ is an adically quasi-coherent almost coherent 
    $\sO_{\fX}$-module.
    Hence, Zariski descent yields  equivalences
	\eq{eq1;thm;finiteness}{ 
	\begin{aligned}
	& R\Gamma(\fX,\G_a(0)_{(X,\tX),\Zar}^{\rm for})\simeq 
 		\lim_{q\in \Delta} \G_a(0)_{(X,\tX),\Zar}^{\rm for}(\fU_{[q]})[0],\\
 		& R\Gamma(\fXrig,\cO(1))\simeq 
 		\lim_{q\in \Delta} R\Gamma(\cU^{\rig}_{[q]},\cO(1)). \end{aligned}}
       Therefore, we get an equivalence
        \[ \Cone(\chi')\simeq \lim_{q\in \Delta}\Cone\big(\G_a(0)_{(X,\tX),\zar}^{\rm for}(\fU_{[q]})[0]\to R\Gamma(\cU^{\rig}_{[q]},\cO(1))\big).    \] 
Moreover, setting $U_{[q]}=\tU_{[q]}[\tfrac{1}{\pi}]$ we have isomorphisms 
        \[\G_a(0)_{(X,\tX)}^{\rm for}(\fU_{[q]})\cong \G_a(0)(U_{[q]}, \tU_{[q]})\otimes_{A_{[q]}}\hat{A}_{[q]}\cong
        \G_a(0)(U_{[q]}, \tU_{[q]})\otimes_{A^{\int}_{[q]}}\widehat{A^{\int}_{[q]}},\]
        where the first isomorphism holds by definition and the second follows from the fact that 
        $\G_a(0)(U_{[q]}, \tU_{[q]})$ is an $A_{[q]}^{\int}$-module and 
        that $\widehat{A_{[q]}^\int}= A_{[q]}^\int\otimes_{A_{[q]}} \hat{A}_{[q]}$, as $A_{[q]}^\int$ is finite over $\tA$, 
        see Lemma \ref{lem:coh}. Together with the isomorphism $\widehat{A_{[q]}^{\rm int}}\cong \cO_{\fXrig}(\fU_{[q]}^{\rig})$,
        see  Lemma \ref{lem:Orig-flat-Ot}, the isomorphism \eqref{cor:Gar4}  yields an isomorphism
        \[\G_a(0)_{(X,\tX),\zar}^{\rm for}(\fU_{[q]}) \cong H^0(\cU^{\rig}_{[q]},\cO(1)\big).\]
        Hence, we conclude that
   \[    \Cone(\chi') \simeq \lim_{q\in \Delta}\tau_{\leq 1}R\Gamma(\cU^{\rig}_{[q]},\cO(1)).
        \]
     Since $U$ is smooth $\cU^{\rig}_{[q]}$ is a smooth affinoid, hence by Theorem \ref{thm.bartenwerfer}, we have $H^i(\cU^{\rig}_{[q]},\cO(1))\in \Mod_R^{0}$, for all $i>0$. 
 	Noting that the limits in \eqref{eq1;thm;finiteness} are computed as finite limits and that $\cD(R)^0\subset \cD(R)$ is closed under finite limits, we conclude.
    \end{proof}
\end{thm}

\subsection{An extension to singular schemes}\label{ss;sing}
\def\NSchKS{\mathrm{NSch}_K^\cP}
\def\NSchKSS{\mathrm{NSch}_K^{\cP,\leq 3}}
\def\cFS{\cF^{\cP}}

The content of this subsection is a result of a private correspondence with B. Bhatt.
In the above proof of Theorem \ref{thm;main}, the smoothness is used only in 
Theorem \ref{thm:fg-uncond}, which depends on Bartenwerfer's theorem \cite[Folgerung~3]{B2} (see Theorem \ref{thm.bartenwerfer}).
Having a generalization to affinoids with certain prescribed singularities, we can relax the smoothness condition to such singularities,  see Theorem \ref{thm;main-sing} below.

Let $\cP$ be a class of singularities for normal schemes separated of finite type over $K$ (e.g. rational, log canonical, Du Bois) and $\NSchKS$ be the category of normal schemes $X$ separated of finite type over $K$ such that singularities of $X$ are at most $\cP$ and that $X_{\mathrm{sing}}$ is proper over $K$.
Consider the functor (cf. \eqref{eq;cFr}):
\[ \NSchKS\ni  X\mapsto \cFS(X) := R\Gamma_t(X/R,\G_a(0)).\]
Consider the following condition:
\begin{enumerate}
\item[$(\dagger)$]
For any proper normal scheme $X$ over $K$ whose singularities are at most $\cP$ and for any affinoid subdomain 
$V\subset X^{\rig}$ and for any $i\in \Z_{>0}$, there exist $c\in K^\times$ with $|c|<1$ such that $c\cdot H^i(V,\cO(1))=0$.
\end{enumerate}
If $(\dagger)$ holds, the conclusions of Theorem \ref{thm:fg-uncond} hold true assuming that the singularities of $X=\tX\otimes_{\OK} K$ is at most $\cP$.
Hence, by Remark \ref{rmk;thm;comparison-final}, we get the  following.

\begin{thm}\label{thm;main-sing}
Assume that $(\dagger)$ holds and $K$ is either a discrete valuation field or algebraically closed.  Then, the following hold:
\begin{enumerate}
\item
We have $\cFS(X)\in \sD^{\geq0}(R)^f$, for every $X\in \NSchKS$.
\item
For a morphism $f:Y\to X$ in $\NSchKS$ which is an isomorphism over a dense open $U\subset X$ such that $X_{\mathrm{sing}}\subset U$ and $Y_{\mathrm{sing}}\subset f^{-1}(U)$, we have $\cF(X)\simeq \cF(Y)$.
\item
Let $X\in \NSchKS$ and assume $\ch(K)=0$ or $\dim X\le 3$.
Then, we have an equivalence
\[\cF(X)\otimes_R K \simeq R\Gamma(\Xb,\cO),\]
for a normal compactification $X\hookrightarrow \Xb$ over $K$ such that  $\Xb$ is regular in a neighborhood of $\Xb-X$.
\end{enumerate}
 \end{thm} 
 
 By the same argument as in the introduction, we can deduce the following from the above theorem.

\begin{thm}\label{thm;application-sing}
Let $k$ be an arbitrary field and let $\Lambda \subseteq k$ be the integral closure of $\bZ$.
Let $X$ be a proper normal scheme over $K$ whose singularities are at most $\cP$. Assume that $(\dagger)$ holds and either $\ch(k)=0$ or $\dim(X)\leq 3$. For a $k$-rational map $\phi: X\dasharrow X$ whose indeterminate locus is disjoint with $X_{\mathrm{sing}}$ and $i\in \N$, the characteristic polynomials ${\rm det}\big(T-\phi^*|H^i(X,\cO_X)\big)$ and ${\rm det}\big(T-\phi^*|H^i(X,\omega_X)\big)$ lie in $\Lambda[T]$.
\end{thm}

\bibliographystyle{alpha}
\bibliography{bib-2}

\end{document}